\documentclass[11pt,a4paper,leqno]{article}
\usepackage{a4wide}
\usepackage{latexsym}
\usepackage{amsmath}
\usepackage{amssymb}

\usepackage{color}

\newtheorem{defin}{Definition}
\newtheorem{lemma}{Lemma}  
\newtheorem{prop}{Proposition}
\newtheorem{theo}{Theorem}

\newenvironment{proof}{\medskip\par\noindent{\bf Proof}}{\hfill $\Box$
\medskip\par}

\def\b{\beta} 
\def\C{\mathbb{C}} 
\def\N{\mathbb{N}}
\def\R{\mathbb{R}}

\begin{document}



\title{On parametric Gevrey asymptotics for singularly perturbed partial differential equations with delays}
\author{Alberto Lastra\footnote{The author is partially supported by the project MTM2012-31439 of Ministerio de Ciencia e Innovacion (Spain), and Caja de Burgos Obra Social}, St\'ephane Malek\footnote{The author is partially supported by the french ANR-10-JCJC 0105 project and the PHC Polonium 2013 project No.28217SG}}

\maketitle

\begin{center}
{\bf Abstract}
\end{center}

We study a family of singularly perturbed $q-$difference-differential equations in the complex domain. We provide sectorial holomorphic solutions in the perturbation parameter $\epsilon$. Moreover, we achieve the existence of a common formal power series in $\epsilon$ which represents each actual solution, and establish $q-$Gevrey estimates involved in this representation. The proof of the main result rests on a new version of the so-called Malgrange-Sibuya Theorem regarding $q-$Gevrey asymptotics. A particular Dirichlet like series is studied on the way.


\begin{section}{Introduction}

We study a family of $q-$difference-differential equations of the form
\begin{equation}\label{e70}
\epsilon \partial_{t}\partial_{z}^{S}X(\epsilon,t,z)+a\partial_{z}^{S}X(\epsilon,t,z)=\sum_{\underline{\kappa}=(\kappa_{0},\kappa_{1})\in\mathcal{N}}b_{\underline{\kappa}}(\epsilon,z)(\partial_{t}^{\kappa_{0}}\partial_{z}^{\kappa_{1}}X)(\epsilon,q^{m_{\underline{\kappa},1}}t,q^{m_{\underline{\kappa},2}}z),
\end{equation}

under appropriate initial conditions

\begin{equation}\label{e74}
(\partial_{z}^{j}X)(\epsilon,t,0)=\phi_{j}(\epsilon,t),\quad 0\le j\le S-1.
\end{equation}

Here, $S$ is an integer with $S\ge1$ , and $a\in\C^{\star}:=\C\setminus\{0\}$.  $\mathcal{N}$ stands for a finite subset of $\N^{2}$, where $\N:=\{0,1,2,...\}$ is the set of nonnegative integers. For every $(\kappa_0,\kappa_1)\in\mathcal{N}$, $b_{\underline{\kappa}}(\epsilon,z)$ turns out to be a polynomial in the variable $z$ with holomorphic and bounded coefficients in a neighborhood of the origin in the parameter, and $m_{\underline{\kappa},1},m_{\underline{\kappa},2}\in\N$.

From now on, $q$ stands for a fixed real number with $0<q<1$.

We construct actual holomorphic solutions $X(\epsilon,t,z)$ for the previous Cauchy problem in $\mathcal{E}\times\mathcal{T}\times\C$, where $\mathcal{E}$ is a bounded open sector in the complex plane with vertex at the origin, and $\mathcal{T}$ is an unbounded well-chosen open set. The procedure is based on the use of the map $t\mapsto t/\epsilon$ which was firstly considered by M. Canalis-Durand, J. Mozo-Fernandez and R. Sch\"{a}fke in~\cite{canalisjorge} to transform a singularly perturbed equation into an auxiliary regularly perturbed equation, easier to handle. This celebrated technique has also been used in the study of singularly perturbed partial differential equations (see \cite{threefold} and \cite{malek0} for example), $q-$difference-differential equations (like in~\cite{malek} or \cite{lastramalek}), and more recently to the study of difference-differential equations (see~\cite{malek3}). 

Indeed, the present work is motivated by a previous work \cite{malek3}, where the second author studies a singularly perturbed difference-differential equation with small delay. This work can be seen as a continuation of that one. The dynamics appearing in that previous work involve a small shift in variable $t$ with respect to $\epsilon$, meaning that they are of the form $(\epsilon,t,z)\mapsto (\epsilon,t+\kappa_2\epsilon,z)$, whereas the actual work deals with a shrinking behaviour in both $t$ and $z$ variables.

In \cite{malek3}, a Gevrey $1+$ phenomenon, with estimates associated to the sequence $\left((\frac{n}{\log n})^n\right)_{n\ge0}$, is observed for the series solution of the problem. This sequence naturally appears when working with difference equations (see~\cite{bfi}, \cite{bf} for example). Now, a $q-$Gevrey like behaviour, related to the sequence of estimates $(q^{-n^2})_{n\ge0}$, appears. This behaviour comes up in the context of $q-$difference equations (see~\cite{divizio},~\cite{ramis}). One can observe that $1+$ sequence is asymptotically upper bounded by Gevrey sequence $(n!)_{n\ge0}$, and this one is upper bounded by $q-$Gevrey sequence $(q^{-n^2})_{n\ge0}$. 


The main aim of this work is to construct actual holomorphic solutions $X(\epsilon,t,z)$ of (\ref{e70})+(\ref{e74}) and obtain sufficient conditions for the existence and unicity of a formal power series in the parameter $\epsilon$, $\hat{X}(\epsilon,t,z)=\sum_{\beta\ge0}\hat{X}_{\beta}(t,z)\frac{\epsilon^{\beta}}{\beta!}$, owing its coefficients in an adequate functional space, and such that $X$ is represented by $\hat{X}$ in a sense to precise (see Theorem~\ref{teorema813}). This representation is measured in terms of $q-$Gevrey bounds due to the appearance of $q-$difference operators on the right-hand side in (\ref{e70}).

The Cauchy problem (\ref{e70})+(\ref{e74}) we consider in this paper comes also within the framework of the asymptotic analysis of linear differential and partial differential equations with multiplicative delays.

In the context of differential equations most of the statements in the literature are dedicated to linear problems of the form
\begin{equation}
 x'(t) = F(t,x(\lambda_{1}t),\ldots,x(\lambda_{n}t),x'(\lambda_{1}t),\ldots,x'(\lambda_{n}t))
\end{equation}
where $F$ are vector valued polynomial functions in $t$ and linear in its other arguments, where $0<\lambda_{j}<1$, for $1 \leq j \leq n$ are real numbers, and concern the study of asymptotic behaviour of some of their solutions $x(t)$ as
$t$ tends to infinity for given initial data $x(0)$. When $F$ is real or matrix valued and with constant coefficients, we quote
\cite{cady}, \cite{fomaocta}, \cite{is1}, \cite{kale}. For polynomial $F$ in $t$, we notice \cite{ce2}, \cite{devo}. For studies in a complex
variable $t$, we refer to \cite{deis}, \cite{zh}. For more general delay functional equations, we indicate \cite{ce1}.

In the framework of linear partial differential equations, we mention a series of papers devoted to general results on the existence and unicity of holomorphic solutions to generalized Cauchy-Kowalevski type problems with shrinkings of the form
$$\partial_{t}^{m}u(t,x) = f(t,x,u(t,x),( \partial_{x}^{l}u(t,x), \partial_{x}^{p}u(\alpha(t)t,x), \partial_{x}^{q}u(t,\beta(t,x)x) )_{(l,p,q) \in I} )
$$
for some integer $m \geq 1$, a finite set $I$, and where $f$ is analytic or of Gevrey type function and such that the functions $\alpha(t)$ and $\beta(t,x)$ satisfy the shrinking constraints $|\alpha(t)|<1$ and $|\beta(t,x)|<1$
for given initial data $(\partial_{t}^{j}u)(0,x)$, $0 \leq j \leq m-1$ that belong to some functional space. We refer to \cite{aulewa}, \cite{ka}, \cite{kaya}. For partial differential problems with contractions dealing with less regular solution spaces like Sobolev spaces, we quote \cite{ro}, for instance.

Let us briefly reproduce the strategy followed. We consider a finite family of sectors with vertex at the origin $(\mathcal{E}_{i})_{0\le i\le \nu}$ which provides a good covering at 0 in the variable $\epsilon$ (see Definition~\ref{goodcovering}). Let $i\in\{0,1,...,\nu-1\}$. One can consider an auxiliary Cauchy problem  
$$\partial_{z}^{S}W(\epsilon,\tau,z)=\sum_{\underline{\kappa}=(\kappa_{0},\kappa_{1})\in\mathcal{N}}\frac{b_{\underline{\kappa}}(\epsilon,z)}{(a-\tau)q^{m_{\underline{\kappa},1}(\kappa_{0}+1)}}\left(-\frac{\tau}{\epsilon}\right)^{\kappa_{0}}\left(\partial_{z}^{\kappa_{1}}W\right)\left(\epsilon,q^{-m_{\underline{\kappa},1}}\tau,q^{m_{\underline{\kappa},2}}z\right),$$ 
with initial conditions $(\partial_{z}^{j}W)(\epsilon,\tau,0)=W_{j}(\epsilon,\tau)$, $0\le j\le S-1$. We assume $W_{j}$ is a holomorphic function in $(D(0,r_0)\setminus\{0\})\times D(0,\hat{R}_0)$ for some $r_0,\hat{R}_{0}>0$, for every $0\le j\le S-1$, which is upper bounded in terms of $q-$Gevrey bounds (see (\ref{e2392})). Moreover, we assume each $W_{j}$ can be extended to $\mathcal{E}_i\times \mathcal{S}$, where $\mathcal{S}$ is a sector with vertex at the origin, and verifying $q-$Gevrey bounds in $\mathcal{E}_{i}\times S_0$, with $S_{0}:=\{z\in\mathcal{S}:|z|\ge R_0\}$ (see (\ref{e239})). Under these hypotheses, one can construct a formal solution to the auxiliary Cauchy problem, $W(\epsilon,\tau,z)=\sum_{\beta\ge0}W_{\beta}(\epsilon,\tau)\frac{z^{\beta}}{\beta!}$, where $W_{\beta}(\epsilon,\tau)$ turns out to be a holomorphic function in $(D(0,r_0)\setminus\{0\})\times (D_{\beta}\setminus\{0\})$. Here, $D_{\beta}$ is a disc centered at the origin with radius decreasing to 0 whenever $\beta$ tends to infinity, and reproducing $q-$Gevrey bounds given by the initial conditions (see (\ref{e469})). Moreover, each $W_{\beta}(\epsilon,\tau)$ can be extended to $\mathcal{E}_{i}\times S_{\beta}$ under $q-$Gevrey bounds (see (\ref{e246})), where $S_{\beta}:=\{z\in\mathcal{S}:|z|>R_{\beta}\}$, with $(R_{\beta})_{\beta\ge0}$ being a sequence of positive numbers that decrease to 0. We assume $S_{\beta}\cap D_{\beta}\neq\emptyset$ for every $\beta\ge0$. The decrease rate of both $R_{\beta}$ and the radius of $D_{\beta}$ has to be chosen adequately, in accordance to the elements of a $q-$Gevrey sequence such as $(q^{\alpha\beta^{2}})_{\beta\ge0}$ for some $\alpha>0$.

The main difficulty in this work is the occurrence of propagation of singularities in the coefficients of the auxiliary problem which leads to a small divisor phenomenon. The singular points form a sequence of complex numbers tending to 0. As a result, one can only obtain a formal solution for the auxiliary problem. In \cite{lastramaleksanz}, a small divisor phenomenon comes from the Fuchsian operator studied in the main Cauchy problem. There, $q\in\C$ is chosen to have $|q|>1$, whilst in the present work $q\in\R$ with $0<q<1$. A suchlike phenomenon also appears in~\cite{threefold}, where the asymptotics in the parameter suffers the effect of a small divisor, and it is solved studying a Dirichlet like series. 

General Dirichlet series of the form 
$$\sum_{n\ge0}a_ne^{-\lambda_nz}$$
have been throughly studied in the case when $(\lambda_n)_{n\ge0}$ is an increasing sequence of real numbers to $\infty$ (see~\cite{hardy}, \cite{tit}, \cite{apostol}) or a sequence of complex numbers with $|\lambda_n|\to\infty$ (see~\cite{leont}). This theory has also been developed when working with almost periodic functions, introduced by H. Bohr (see~\cite{bohr}, \cite{besi}, \cite{cord}), which are the uniform limits in $\R$ of exponential polynomials $\sum_{k=1}^{n}a_ke^{is_kx}$, where the values $s_k$ belong to the so-called spectrum $\Lambda\subseteq\R$. However, we are more interested in the behaviour of the sum when $x$ tends to $\infty$ in he positive imaginary axis. Our technique rests on Euler-Mac-Laurin formula, Watson's Lemma and the equivalence between null $q-$Gevrey asymptotics and the fact of being $q-$exponentially small.

In~\cite{threefold}, we solve the problem by means of a Dirichlet series with a spectrum being of the form $(\frac{1}{(k+1)^{\alpha}})_{k\ge0}$. Now, the spectrum which helps us to achieve our purpose is of geometric nature (see Lemma~\ref{lema646}).

The growth properties of $W_{\beta}$ for $\beta\ge0$ allow us to apply a Laplace like transform on each of them with respect to the variable $\tau$ in order to provide a holomorphic solution $X_{i}(\epsilon,\tau,z)$ of the main problem, defined in $\mathcal{E}_{i}\times\mathcal{T}\times\C$, for some appropriate unbounded open set $\mathcal{T}$. In addition to this, one has null $q-$Gevrey asymptotic bounds for the difference of $X_{i}$ and $X_{i+1}$ when the domain of the variable $z$ is restricted to a bounded set, meaning that for every $\rho>0$, there exist $L_1,L_2>0$ such that 
$$\sup_{\stackrel{t\in\tau}{z\in D(0,\rho)}}\left|X_{i+1}(\epsilon,t,z)-X_{i}(\epsilon,t,z)\right|\le  L_1e^{-L_2\frac{1}{(-\log(q))2}\log^{2}|\epsilon|},$$
for every $\epsilon\in\mathcal{E}_{i}\cap\mathcal{E}_{i+1}$.

Finally, a novel version regarding $q-$Gevrey asymptotics of Malgrange-Sibuya Theorem (Theorem~\ref{teoremams}) leads us to the main result in the present work (Theorem~\ref{teorema813}), where we guarantee the existence of a formal power series in $\epsilon$, 
$$\hat{X}(\epsilon,t,z)=\sum_{\beta\ge0}\frac{X_{\beta}(t,z)}{\beta!}\epsilon^{\beta}\in\mathbb{H}_{\mathcal{T},\rho}[[\epsilon]],$$
with coefficients in the Banach space of bounded holomorphic functions defined in $\mathcal{T}\times D(0,\rho)$, which is common for every $0\le i\le \nu-1$, and such that $X_{i}$ admits $\hat{X}$ as its $q-$Gevrey asymptotic expansion of some positive type in he variable $\epsilon$ (see (\ref{e839})).

It is worth pointing out that a $q-$Gevrey version of Malgrange-Sibuya Theorem was already obtained in \cite{lastramalek}, when dealing with $q\in\C$, $|q|>1$. There, the type in the asymptotic expansion involved suffers some increasement. This is so due to the need of extension results in ultradifferentiable classes of functions (see~\cite{bonetbraunmeisetaylor},~\cite{chch}) to be applied along the proof. Here, the geometry of the problem changes so that we are able to maintain the type $q-$Gevrey. The proof rests on the classical Malgrange-Sibuya Theorem (see~\cite{hssi}). 

The paper is organized as follows.

In Section~\ref{seccion1} and Section~\ref{seccion2}, we introduce Banach spaces of formal power series in order to solve auxiliary Cauchy problems with the help of fixed point results involving complete metric spaces. In Section~\ref{seccion1}, this result is achieved when dealing with formal power series with holomorphic coefficients in a product of a finite sector with vertex at the origin times an infinite sector, while in Section~\ref{seccion2} the result is obtained when dealing with a product of two punctured discs at 0.

In Section~\ref{seccion3}, we first recall the definition and main properties of a Laplace like transform, and $q-$Gevrey asymptotic expansions (Subsection~\ref{subseccion31}). Next, we construct analytic solutions for the main problem and determine flat $q-$Gevrey bounds for the difference of two solutions when the intersection of the domains in the perturbation parameter is not empty (Subsection~\ref{subseccion32}). In the proof, a Dirichlet type series is studied. The section is concluded proving the existence of a formal power series in the perturbation parameter which represents every solution in some sense which is specified (Subsection~\ref{subseccion33}).

\end{section}

\begin{section}{A Cauchy problem in weighted Banach spaces of Taylor power series}\label{seccion1}

$M,A_1,C,\delta_{1}>0$ are fixed positive real numbers throughout the present work. Let $q\in\R$ with $0<q<1$ and $(R_{\beta})_{\beta\ge0}$ be a sequence of positive real numbers.

We consider an open and bounded sector $\mathcal{E}$ with vertex at the origin and we fix an open and unbounded sector $\mathcal{S}$ with vertex at the origin having positive distance to a fixed complex number $a\in\C^{\star}$, it is to say, there exists $M_{1}>0$ such that $|\tau-a|>M_{1}$ for every $\tau\in \mathcal{S}$. We write $S_{\beta}$ for the subset of $\mathcal{S}$ defined by
$$S_{\beta}:=\left\{z\in \mathcal{S}: |z|>R_{\beta}\right\}.$$

The incoming definition of Banach spaces of functions and formal power series turns out to be an adaptation of the corresponding one in~\cite{lastramalek}. Here, the symmetry of these norms at 0 and the point of infinity in the $\tau$ variable has to be  removed, so that a Laplace like transform of the elements in these Banach spaces makes sense.

\begin{defin}
Let $\epsilon\in\mathcal{E}$ and $\beta\in\N$. $E_{\beta,\epsilon,S_{\beta}}$ denotes the vector space of functions $v\in\mathcal{O}(S_{\beta})$ such that
$$\left\|v(\tau)\right\|_{\beta,\epsilon,S_{\beta}}:=\sup_{\tau\in S_{\beta}}\left\{\frac{|v(\tau)|}{e^{M\log^{2}\left(\frac{|\tau|}{|\epsilon|}+\delta_1\right)}}\left|\frac{\tau}{\epsilon}\right|^{-C\beta}\right\}q^{-A_{1}\beta^{2}}$$
is finite.

Let $\delta>0$. $H(\epsilon,\delta,\mathcal{S})$ denotes the complex vector space of all formal power series $v(\tau,z)=\sum_{\beta\ge0}v_{\beta}(\tau)\frac{z^{\beta}}{\beta!}$ with $v_{\beta}\in\mathcal{O}(S_{\beta})$ for every $\beta\ge0$ and such that
$$ \left\|v(\tau,z)\right\|_{(\epsilon,\delta,\mathcal{S})}:=\sum_{\beta\ge0}\left\|v_{\beta}(\tau)\right\|_{\beta,\epsilon,S_{\beta}}\frac{\delta^{\beta}}{\beta!}<\infty.$$
It is straightforward to check that the pair $(H(\epsilon,\delta,\mathcal{S}),\left\|\cdot\right\|_{(\epsilon,\delta,\mathcal{S})})$ is a Banach space.
\end{defin}

For our purposes, the elements in the sequence $(R_{\beta})_{\beta\ge0}$ are chosen to be related to the ones in a $q-$Gevrey sequence. This choice would provide that $S_{\beta}$ tends to $\mathcal{S}$ when $\beta\to\infty$.

Let $(\mathbb{E}_{\beta})_{\beta\ge0}$ be a family of complex functional Banach spaces.
For every $v(\tau,z)=\sum_{\beta\ge0}v_{\beta}(\tau)\frac{\tau^{\beta}}{\beta!}\in\left(\cup_{\beta\ge0}\mathbb{E}_{\beta}\right)[[z]]$, we consider the formal integration operator $\partial_{z}^{-1}$ defined on $\left(\cup_{\beta\ge0}\mathbb{E}_{\beta}\right)[[z]]$ by
$$\partial_{z}^{-1}(v(\tau,z)):=\sum_{\beta\ge1}v_{\beta-1}(\tau)\frac{z^{\beta}}{\beta!}.$$

\begin{lemma}\label{lema130}
Let $s,\ell_{0},\ell_{1},m_{1},m_{2}\in\N$, $\delta>0$ and $\epsilon\in\mathcal{E}$. We assume that

\begin{equation}\label{e129}
C(\ell_{1}+s)-\ell_{0}-2m_{1}M(-\log(q)) \ge0. 
\end{equation}
In addition to this, we consider the elements in $(R_{\beta})_{\beta\ge0}$ are such that
\begin{equation}\label{e130}
R_{\beta}\ge q^{m_{1}}R_{\beta-\ell_{1}-s},
\end{equation}
for every $\beta\ge \ell_{1}+s$. Moreover, we assume there exist constants  $d_{1},d_{2}>0$ such that
\begin{equation}\label{e134} 
R_{\beta}\ge d_{1}q^{d_{2}\beta},
\end{equation}
for every $\beta\ge0$. In addition to this, we assume
\begin{equation}\label{e135}
m_2-2A_1(\ell_1+s)-m_1C+d_2\left[\ell_0-2m_{1}M\log(q)-C(\ell_1+s)\right]> 0.
\end{equation}

Under the previous assumptions, there exists a positive constant $C_{11}$, which does not depend on $\epsilon$ nor $\delta$, such that
$$\left\|z^{s}\left(-\frac{\tau}{\epsilon}\right)^{\ell_0}\frac{1}{q^{m_{1}(\ell_0+1)}}(\partial_{z}^{-\ell_1}v)(\tau q^{-m_1},zq^{m_2})\right\|_{(\epsilon,\delta,\mathcal{S})}\le C_{11}\delta^{\ell_{1}+s}\left\|v(\tau,z)\right\|_{(\epsilon,\delta,\mathcal{S})},$$
for every $v\in H(\epsilon,\delta,\mathcal{S})$.
\end{lemma}
\begin{proof}
Let $v(\tau,z)=\sum_{\beta\ge0}v_{\beta}(\tau)\frac{z^{\beta}}{\beta!}\in H(\epsilon,\delta,\mathcal{S})$. We have that 
\begin{multline}\left\|z^{s}\left(-\frac{\tau}{\epsilon}\right)^{\ell_0}\frac{1}{q^{m_{1}(\ell_0+1)}}(\partial_{z}^{-\ell_1}v)(\tau q^{-m_1},zq^{m_2})\right\|_{(\epsilon,\delta,\mathcal{S})}\\
=\left\|\sum_{\beta\ge\ell_{1}+s}v_{\beta-\ell_{1}-s}(\tau q^{-m_{1}})q^{m_{2}(\beta-s)-m_{1}(\ell_{0}+1)}\frac{\beta!}{(\beta-s)!}\left(-\frac{\tau}{\epsilon}\right)^{\ell_{0}}\frac{z^{\beta}}{\beta!}\right\|_{(\epsilon,\delta,\mathcal{S})}.\label{e143}
\end{multline}

From (\ref{e130}), one derives that for every $\tau\in S_{\beta}$, $v_{\beta-\ell_{1}-s}(\tau q^{-m_{1}})$ is well defined and the function $\tau \mapsto v_{\beta-\ell_{1}-s}(q^{-m_{1}}\tau)$ is holomorphic in $S_{\beta}$ for every $\beta\ge \ell_{1}+s$. The expression in (\ref{e143}) equals
\begin{equation}\label{e147}
\sum_{\beta\ge\ell_{1}+s}\left\|v_{\beta-\ell_{1}-s}(\tau q^{-m_{1}})q^{m_{2}(\beta-s)-m_{1}(\ell_{0}+1)}\frac{\beta!}{(\beta-s)!}\left(-\frac{\tau}{\epsilon}\right)^{\ell_{0}}\right\|_{\beta,\epsilon,S_{\beta}}\frac{\delta^{\beta}}{\beta!}.
\end{equation}
Let $\beta\ge \ell_{1}+s$. From the definition of the norm $\left\|\cdot\right\|_{\beta,\epsilon,S_{\beta}}$, we get
$$\left\|v_{\beta-\ell_{1}-s}(\tau q^{-m_{1}})q^{m_{2}(\beta-s)-m_{1}(\ell_{0}+1)}\frac{\beta!}{(\beta-s)!}\left(-\frac{\tau}{\epsilon}\right)^{\ell_{0}}\right\|_{\beta,\epsilon,S_{\beta}}$$
\begin{align}
&=\sup_{\tau\in S_{\beta}}\left\{|v_{\beta-\ell_{1}-s}(\tau q^{-m_{1}})|\left(\frac{|\tau|q^{-m_{1}}}{|\epsilon|}\right)^{-C(\beta-\ell_{1}-s)}e^{-M\log^{2}\left(\frac{|\tau|}{|\epsilon|}q^{-m_{1}}+\delta_{1}\right)}\right.\nonumber\\
&\left.\times\left|\frac{\tau}{\epsilon}\right|^{\ell_0}\left(\frac{|\tau|q^{-m_1}}{|\epsilon|}\right)^{C(\beta-\ell_1-s)}e^{M\log^{2}\left(\frac{|\tau|q^{-m_1}}{|\epsilon|}+\delta_1\right)}\left|\frac{\tau}{\epsilon}\right|^{-C\beta}e^{-M\log^{2}(\frac{|\tau|}{|\epsilon|}+\delta_1)}\right\}\nonumber\\
&\times \frac{\beta!}{(\beta-s)!}q^{m_2(\beta-s)-m_1(\ell_0+1)}q^{-A_1\beta^2}q^{A_{1}(\beta-\ell_{1}-s)^{2}}q^{-A_{1}(\beta-\ell_{1}-s)^{2}}.\label{e159}
\end{align}

It is immediate to check that
\begin{equation}\label{e160}
e^{-M\log^2(\frac{|\tau|}{|\epsilon|}+\delta_1)+M\log^2\left(\frac{|\tau|q^{-m_1}}{|\epsilon|}+\delta_1\right)}\le C_{01}e^{-M\log^{2}\left(\frac{|\tau|}{|\epsilon|}\right)+M\log^{2}\left(\frac{|\tau|q^{-m_{1}}}{|\epsilon|}\right)}\le C_{02}\left(\frac{|\tau|}{|\epsilon|}\right)^{-2m_{1}M\log(q)},
\end{equation}
for some positive constants $C_{01}$ and $C_{02}$ only depending on $q,m_{1},M$. Moreover, 
$$(|\tau|q^{-m_1})^{C(\beta-\ell_1-s)}=C_{03} q^{-m_1C\beta}|\tau|^{C(\beta-\ell_1-s)},$$
for some constant $C_{03}>0$ depending on $q,m_1,\ell_1,s$. This last equality and (\ref{e129}) yield
$$\left|\frac{\tau}{\epsilon}\right|^{\ell_0}\left(\frac{|\tau|q^{-m_1}}{|\epsilon|}\right)^{C(\beta-\ell_1-s)}\left|\frac{\tau}{\epsilon}\right|^{-C\beta}\left(\frac{|\tau|}{|\epsilon|}\right)^{-2m_1M\log(q)}$$
\begin{align*}
&=C_{04}\left(\frac{|\epsilon|}{|\tau|}^{-\ell_0+C(\ell_1+s)+2m_1M\log(q)}\right)\\
&\le C_{05}|\tau|^{\ell_0-2m_1M\log(q)-C(\ell_1+s)}q^{-m_1C\beta},
\end{align*}
for some positive constants $C_{04}$ and $C_{05}$ depending on $q,m,\ell_0,\ell_1,s,C$ and $\mathcal{E}$. From the hypothesis (\ref{e134}) on $R_{\beta}$, the last expression is upper bounded by $$C_{05}q^{(-m_1C+d_2(\ell_0-2m_1M\log(q)-C(\ell_1+s)))\beta},$$
for some positive constant $C_{05}$ only depending on $q$, $m_{1}$, $\ell_0$, $\ell_1$, $s$, $\mathcal{E}$, $C$ and $d_{1}$. 
Now, from (\ref{e135}) one gets that $\beta!/(\beta-s)!q^{p_1(\beta)}$ is upper bounded by a constant $C_{06}>0$ which does not depend on $\beta$, where 
$$p_{1}(\beta)=m_2(\beta-s)-m_1(\ell_0+1)-A_{1}\beta^{2}+A_{1}(\beta-\ell_1-s)^2-m_1C\beta+d_2\left[\ell_0-2m_1M\log(q)-C(\ell_1+s))\right].$$
Taking into account all these computations, one achieves that (\ref{e159}) can be upper bounded by 
$$C_{05}\sup_{\tau\in S_{\beta}}\left\{|v_{\beta-\ell_1-s}(\tau q^{-m_{1}})|\left(\frac{|\tau|q^{-m_1}}{|\epsilon|}\right)^{-C(\beta-\ell_1-s)}e^{-M\log^{2}(\frac{|\tau|q^{-m_1}}{|\epsilon|}+\delta_1)}\right\}q^{-A_{1}(\beta-\ell_1-s)}.$$
The lemma follows bearing in mind (\ref{e130}) and the definition of the norms in $E_{\beta-(\ell_1+s),\epsilon, S_{\beta-(\ell_1+s)}}$ and of $H(\epsilon,\delta,\mathcal{S})$. 
\end{proof}

\textbf{Remark:} The hypotheses made in (\ref{e130}), (\ref{e134}) and (\ref{e135}) are veryfied if one departs from $R_{\beta}=d_1q^{d_2\beta}$ for some small enough positive $d_2$, and any $d_1>0$, provided (\ref{e129}) is satisfied and $m_{2}-2A_{1}(\ell_1+s)-m_1C>0$.

\begin{lemma}\label{lema207}
Let $F(\epsilon,\tau)$ be a holomorphic and bounded function defined on $\mathcal{E}\times \mathcal{S}$. Then, there exists a constant $C_{12}=C_{12}(F,\mathcal{E},\mathcal{S})>0$ such that
$$\left\|F(\epsilon,\tau)v_{\epsilon}(\tau,z)\right\|_{(\epsilon,\delta,\mathcal{S})}\le C_{12}\left\|v_{\epsilon}(\tau,z)\right\|_{(\epsilon,\delta,\mathcal{S})}$$
for every $\epsilon\in \mathcal{E}$, every $\delta>0$ and all $v_{\epsilon}\in H(\epsilon,\delta,\mathcal{S})$.
\end{lemma}
\begin{proof}
Direct calculations on the definition of the norms in the space $H(\epsilon,\delta,\mathcal{S})$ allow us to conclude when taking $C_{12}:=\max\{|F(\epsilon,\tau)|:\epsilon\in\mathcal{E},\tau\in \mathcal{S}\}$.
\end{proof}

Let $S\ge1$, and $\mathcal{N}$ be a finite subset of $\N^{2}$. We also fix $a\in\C\setminus\R_{+}$, where $\R_{+}$ stands for the set $\{z\in\C: \hbox{Re}(z)\ge 0,\hbox{ Im}(z)=0\}$.

For every $\underline{\kappa}=(\kappa_{0},\kappa_{1})\in\mathcal{N}$, let $m_{\underline{\kappa},1}$, $m_{\underline{\kappa},2}$ be nonnegative integers and $b_{\underline{\kappa}}(\epsilon,z)\in\mathcal{O}(D(0,r_0))[z]$, where $r_0>0$ is such that $\overline{\mathcal{E}}\subseteq D(0,r_0)$. We write $b_{\underline{\kappa}}(\epsilon,z)=\sum_{s\in I_{\underline{\kappa}}}b_{\underline{\kappa},s}(\epsilon)z^{s}$, where $I_{\underline{\kappa}}$ is a finite subset of $\N$ for every $\underline{\kappa}\in\mathcal{N}$. We assume that $1\le\kappa_{1}<S$ for every $\underline{\kappa}=(\kappa_0,\kappa_1)\in\mathcal{N}$.

We consider the functional equation
\begin{equation}\label{e214}
\partial_{z}^{S}W(\epsilon,\tau,z)=\sum_{\underline{\kappa}=(\kappa_{0},\kappa_{1})\in\mathcal{N}}\frac{b_{\underline{\kappa}}(\epsilon,z)}{(a-\tau)q^{m_{\underline{\kappa},1}(\kappa_{0}+1)}}\left(-\frac{\tau}{\epsilon}\right)^{\kappa_{0}}\left(\partial_{z}^{\kappa_{1}}W\right)\left(\epsilon,q^{-m_{\underline{\kappa},1}}\tau,q^{m_{\underline{\kappa},2}}z\right)
\end{equation}
with initial conditions
\begin{equation}\label{e221}
\left(\partial_{z}^{j}W\right)(\epsilon,\tau,0)=W_{j}(\epsilon,\tau)\quad, 0\le j\le S-1,
\end{equation}
where the function $(\epsilon,\tau)\mapsto W_{j}(\epsilon,\tau)$ is an element in $\mathcal{O}(\mathcal{E}\times\mathcal{S})$ for every $0\le j\le S-1$. 

We make the following 

\textbf{Assumption (A)} For every $\underline{\kappa}=(\kappa_{0},\kappa_{1})\in\mathcal{N}$ and every $s\in I_{\underline{\kappa}}$, we assume
$$ C (S-\kappa_{1}+s)-\kappa_{0}-2m_{\underline{\kappa},1}M(-\log(q))\ge0, $$
$$ \left[ C (S-\kappa_{1}+s)-\kappa_{0}-2m_{\underline{\kappa},1}M(-\log(q))\right]d_2<m_{\underline{\kappa},2}-2A_{1}(S-\kappa_{1}+s)-m_{\underline{\kappa},1}C. $$

\textbf{Assumption (B)} $R_{\beta}\ge q^{m_{\underline{\kappa},1}}R_{\beta-\kappa_1-s}$, and there exist $d_{1},d_{2}>0$ with $R_{\beta}\ge d_{1}q^{d_{2}\beta}$, for every $\underline{\kappa}=(\kappa_{0},\kappa_{1})\in\mathcal{N}$ and every $s\in I_{\underline{\kappa}}$.

\begin{theo}\label{teorema240} 
Let Assumption (A) and Assumption (B) be fulfilled. We assume that the initial conditions in (\ref{e221}) verify there exist $\Delta>0$ and $0<\tilde{M}<M$ such that for every $0\le j\le S-1$
\begin{equation}\label{e239}
|W_{j}(\epsilon,\tau)|\le\Delta e^{\tilde{M}\log^{2}(\frac{|\tau|}{|\epsilon|}+\delta_1)}|\epsilon|^{K_0},
\end{equation} 
for every $\tau\in S_{0}$, $\epsilon\in\mathcal{E}$, where $K_0=\max\{\kappa_0: (\kappa_0,\kappa_1)\in\mathcal{N}\}$. Then, there exists  $W(\epsilon,\tau,z)=\sum_{\beta\ge0}W_{\beta}(\epsilon,\tau)\frac{z^{\beta}}{\beta!}$, formal solution of (\ref{e214})+(\ref{e221}), where $W_{\beta}\in\mathcal{O}(\mathcal{E}\times S_{\beta})$.

Then, there exist positive constants $C_{13}$, and $C_{14}$ (only depending on $q,d_1,d_2,C,S,\delta_1,A_1$), and $\delta>0$ such that
\begin{equation}\label{e246}
|W_{\beta}(\epsilon,\tau)|\le C_{13}\beta!\left(\frac{C_{14}}{\delta}\right)^{\beta}e^{M\log^2(\frac{|\tau|}{|\epsilon|}+\delta_1)}\left|\frac{\tau}{\epsilon}\right|^{C\beta}q^{A_1\beta^{2}},
\end{equation}
for every $\beta\ge0$, all $\epsilon\in\mathcal{E}$ and every $\tau\in S_{\beta}$. 
\end{theo}

\begin{proof}
Let $\epsilon\in\mathcal{E}$. We put $\mathbb{E}:=\{\mathcal{O}(S_{\beta}):\beta\ge0\}$ and define the map $\mathcal{A}_{\epsilon}$  from $\mathbb{E}[[z]]$ into itself by
\begin{align}
\mathcal{A}_{\epsilon}(\tilde{W}(\tau,z)):=\sum_{\underline{\kappa}=(\kappa_{0},\kappa_{1})\in\mathcal{N}}\frac{b_{\underline{\kappa}}(\epsilon,z)}{(a-\tau)q^{m_{\underline{\kappa},1}(\kappa_0+1)}}\left(-\frac{\tau}{\epsilon}\right)^{\kappa_{0}}\left[(\partial_{z}^{\kappa_{1}-S}\tilde{W})\left(q^{-m_{\underline{\kappa},1}}\tau,q^{m_{\underline{\kappa},2}}z\right)\right.\nonumber\\
\left.+\partial_{z}^{\kappa_{1}}w_{\epsilon}(q^{-m_{\underline{\kappa},1}}\tau,q^{m_{\underline{\kappa},2}}z)\right],\label{e243}
\end{align}
where $w_{\epsilon}(\tau,z)=\sum_{j=0}^{S-1}W_{j}(\epsilon,\tau)\frac{z^{j}}{j!}$. For an appropriate choice of $\delta,\Delta>0$, the map $\mathcal{A}_{\epsilon}$ turns out to be a Lipschitz shrinking map.

\begin{lemma}\label{lema264}
There exist $R,\delta,\Delta>0$ (not depending on $\epsilon$) such that:
\begin{enumerate}
\item $\left\|\mathcal{A}_{\epsilon}(\tilde{W}(\tau,z))\right\|_{(\epsilon,\delta,\mathcal{S})}\le R$ for every $\tilde{W}(\tau,z)\in B(0,R)$. $B(0,R)$ denotes the closed ball centered at 0 with radius $R$ in $H(\epsilon,\delta,\mathcal{S})$.
\item $$\left\|\mathcal{A}_{\epsilon}(\tilde{W}_{1}(\tau,z))-\mathcal{A}_{\epsilon}(\tilde{W}_{2}(\tau,z))\right\|_{(\epsilon,\delta,\mathcal{S})}\le\frac{1}{2}\left\|\tilde{W}_{1}(\tau,z)-\tilde{W}_{1}(\tau,z)\right\|_{(\epsilon,\delta,\mathcal{S})}$$
for every $\tilde{W}_{1}$, $\tilde{W}_{2}\in B(0,R)$.
\end{enumerate}
\end{lemma}
\begin{proof}
Let $R>0$ and $\delta>0$. In order to prove the first enunciate, we take $\tilde{W}(\tau,z)\in B(0,R)\subseteq H(\epsilon,\delta,\mathcal{S})$. From Lemma~\ref{lema130} and Lemma~\ref{lema207} we deduce that
\begin{multline}\label{e274}
\left\|\mathcal{A}_{\epsilon}(\tilde{W}(\tau,z)\right\|_{(\epsilon,\delta,\mathcal{S})}\le \sum_{\underline{\kappa}=(\kappa_{0},\kappa_1)\in\mathcal{N}}\sum_{s\in I_{\underline{\kappa}}}\frac{M_{\underline{\kappa}s}}{M_{1}}\left[C_{01}\delta^{S-\kappa_{1}+s}\left\|\tilde{W}(\tau,z)\right\|_{(\epsilon,\delta,\mathcal{S})}\right.\\
\left.+\left\|\frac{z^{s}}{q^{m_{\underline{\kappa},1}(\kappa_0+1)}}\left(-\frac{\tau}{\epsilon}\right)^{\kappa_0}\partial^{\kappa_{1}}w_{\epsilon}(q^{-m_{\underline{\kappa},1}}\tau,q^{m_{\underline{\kappa},2}}z)\right\|_{(\epsilon,\delta,\mathcal{S})}\right],
\end{multline}
with $M_{\underline{\kappa}s}=\sup_{\epsilon\in\mathcal{E}}|b_{\underline{\kappa}s}(\epsilon)|<\infty$ for every $\underline{\kappa}\in\mathcal{N}$ and $s\in I_{\underline{\kappa}}$.

Let us fix $\underline{\kappa}=(\kappa_{0},\kappa_{1})\in\mathcal{N}$ and $s\in I_{\underline{\kappa}}$. Taking into account the definition of $H(\epsilon,\delta,\mathcal{S})$), we derive
$$\left\|\frac{z^{s}}{q^{m_{\underline{\kappa},1}(\kappa_0+1)}}\left(-\frac{\tau}{\epsilon}\right)^{\kappa_0}\partial^{\kappa_{1}}w_{\epsilon}(q^{-m_{\underline{\kappa},1}}\tau,q^{m_{\underline{\kappa},2}}z)\right\|_{(\epsilon,\delta,\mathcal{S})}$$
\begin{align}
&=\sum_{j=s}^{S-1-\kappa_1-s}\left\|W_{j+\kappa_1-s}(\epsilon,q^{-m_{\underline{\kappa},1}}\tau)\left(-\frac{\tau}{\epsilon}\right)^{\kappa_0}\right\|_{j,\epsilon,S_{j}}q^{m_{\underline{\kappa},2}(j-s)-m_{\underline{\kappa},1}(\kappa_0+1)}\frac{j!}{(j-s)!}\frac{\delta^{j}}{j!}\nonumber\\
&\le C_{14}\sum_{j=s}^{S-1-\kappa_1-s}\sup_{\tau\in S_{j}}\left|W_{j+\kappa_1-s}(\epsilon,q^{-m_{\underline{\kappa},1}}\tau)\right|\left|\frac{\tau}{\epsilon}\right|^{\kappa_{0}}\left|\frac{\tau}{\epsilon}\right|^{-Cj}e^{-M\log^{2}(\frac{|\tau|}{|\epsilon|}+\delta_1)}\delta^{j},\label{e272}
\end{align}
for some $C_{14}>0$ which only depends on the parameters defining equation (\ref{e214}). The terms of the form $|\epsilon|^{Cj}$ in the previous expression can be upper bounded by an adequate constant. Taking into account (\ref{e239}), usual estimates in (\ref{e272}) derive 
$$\left\|\mathcal{A}_{\epsilon}(\tilde{W}(\tau,z))\right\|_{(\epsilon,\delta,\mathcal{S})}\le \sum_{\underline{\kappa}=(\kappa_{0},\kappa_1)\in\mathcal{N}}\sum_{s\in I_{\underline{\kappa}}}\frac{M_{\underline{\kappa}s}}{M_{1}}\left[C_{01}\delta^{S-\kappa_{1}+s}\left\|\tilde{W}(\tau,z)\right\|_{(\epsilon,\delta,\mathcal{S})}+C_{15}\right],$$
for some $C_{15}$ depending on the parameters defining the equation, and such that tends to 0 whenever both $\Delta$ and $\delta$ tend to 0. An appropriate choice for these constants allow us to conclude the first part of the proof. 

The second part of the lemma follows similar arguments as before. Let $\tilde{W}_{1}$, $\tilde{W}_{2}\in B(0,R)\subseteq H(\epsilon,\delta,\mathcal{S})$. One has
$$\left\|\mathcal{A}_{\epsilon}(\tilde{W}_{1})-\mathcal{A}_{\epsilon}(\tilde{W}_2)\right\|_{(\epsilon,\delta,\mathcal{S})}\le\sum_{\underline{\kappa}=(\kappa_{0},\kappa_{1})\in\mathcal{N}}\sum_{s\in I_{\underline{\kappa}}}\frac{M_{\underline{\kappa} s}}{M_{1}}C_{01}\delta^{S-\kappa_1+s}\left\|\tilde{W}_{1}-\tilde{W}_{2}\right\|_{(\epsilon,\delta,\mathcal{S})}.$$
The result is achieved with an adequate choice of $\delta>0$.
\end{proof}

Let $R$, $\Delta$ and $\delta$ be as in the previous lemma. Bearing in mind Lemma~\ref{lema264} one can apply the shrinking map theorem on complete metric spaces to guarantee the existence of a fixed point for $\mathcal{A}_{\epsilon}$ in $B(0,R)\subseteq H(\epsilon,\delta,\mathcal{S})$, say $\tilde{W}_{\epsilon}$, which verifies $\left\|\tilde{W}_{\epsilon}(\tau,z)\right\|_{(\epsilon,\delta,\mathcal{S})}\le R$, and $\mathcal{A}_{\epsilon}(\tilde{W}_{\epsilon}(\tau,z))=\tilde{W}_{\epsilon}(\tau,z)$. Let us define
\begin{equation}\label{e300}
W_{\epsilon}(\tau,z)=\partial_{z}^{-S}\tilde{W}_{\epsilon}(\tau,z)+w_{\epsilon}(\tau,z).
\end{equation}
We put  $\tilde{W}(\epsilon,\tau,z):=\tilde{W}_{\epsilon}(\tau,z)$, and $W(\epsilon,\tau,z):=\partial_{z}^{-S}\tilde{W}(\epsilon,\tau,z)+w_{\epsilon}(\tau,z)$. Then, $W(\epsilon,\tau,z)$ can be written as a formal power series in $z$, 
$$W(\epsilon,\tau,z)=\sum_{\beta\ge0}W_{\beta}(\epsilon,\tau)\frac{z^{\beta}}{\beta!},$$
where $W_{\beta+S}(\epsilon,\tau)=\tilde{W}_{\beta,\epsilon}(\tau)$ for every $\beta\ge0$. 

From the construction of $W(\epsilon,\tau,z)$, we have $W(\epsilon,\tau,z)=\sum_{\beta\ge0}W_{\beta}(\epsilon,\tau)\frac{z^{\beta}}{\beta!}$ is a formal solution of (\ref{e214})+(\ref{e221}). Moreover, from the domain of holomorphy of the initial conditions in (\ref{e221}) and the recursion formula satisfied by the coefficients in $W(\epsilon,\tau,z)$:
\begin{equation}\label{e305} \frac{W_{h+S}(\epsilon,\tau)}{h!}=\sum_{\underline{\kappa}=(\kappa_0,\kappa_1)\in\mathcal{N}}\sum_{h_{1}+h_{2}=h,h_{1}\in I_{\underline{\kappa}}}b_{\underline{\kappa},h_{1}}(\epsilon)\left(-\frac{\tau}{\epsilon}\right)^{\kappa_{0}}\frac{q^{m_{\underline{\kappa},2}h_{2}}}{(a-\tau)h_{2}!q^{m_{\underline{\kappa},1}(\kappa_0+1)}}W_{h_{2}+\kappa_{1}}(\epsilon,q^{-m_{\underline{\kappa},1}}\tau),
\end{equation}
we can conclude the function $(\epsilon,\tau)\mapsto W_{\beta}\in\mathcal{O}(\mathcal{E}\times \mathcal{S})$ for every $\beta\ge 0$. 

Finally, the estimates in (\ref{e246}) are obtained for every $\beta\ge0$ from the fact that $\tilde{W}_{\epsilon}\in B(0,R)\subseteq H(\epsilon,\delta,\mathcal{S})$. The definition of the elements in $H(\epsilon,\delta,\mathcal{S})$ lead us to $$\left\|\tilde{W}_{\beta,\epsilon}\right\|_{\beta,\epsilon,S_{\beta}}\le R\beta!\left(\frac{1}{\delta}\right)^{\beta},$$
so that
$$|W_{\beta}(\epsilon,\tau)|=|\tilde{W}_{\beta-S,\epsilon}(\tau)|\le R(\beta-S)!\left(\frac{1}{\delta}\right)^{\beta-S}e^{M\log^{2}\left(\frac{|\tau|}{|\epsilon|}+\delta_1\right)}\left|\frac{\tau}{\epsilon}\right|^{C(\beta-S)}q^{A_{1}(\beta-S)^{2}},$$
for every $\beta\ge S$. In addition to this, Assumption (B) and usual estimates allow us to refine the previous estimates leading to
$$|W_{\beta}(\epsilon,\tau)|\le C_{13}\beta!\left(\frac{C_{14}}{\delta}\right)^{\beta}e^{M\log^2(\frac{|\tau|}{|\epsilon|}+\delta_1)}\left|\frac{\tau}{\epsilon}\right|^{C\beta}q^{A_1\beta^{2}},$$
for some constants $C_{13}>0$ and $C_{14}>0$ which only depend on $q,d_1,d_2,C,S,\delta_1$ and $A_1$. This is valid for every $\epsilon\in\mathcal{E}$ and $\tau\in S_{\beta}$. The hypothesis (\ref{e239}) in the enunciate allows us to affirm that (\ref{e246}) is also valid for $0\le \beta\le S-1$.

\end{proof}

\textbf{Remark:} One derives holomorphy of $W_{\beta}$ in the variable $\tau$ in the whole sector $\mathcal{S}$, and not only in $S_{\beta}$ for every $\beta\ge S$ whilst the estimates are only given for $\tau\in S_{\beta}$. It is also worth saying that $R>0$ can be arbitrarily chosen whenever $s>0$ for every $s\in I_{\underline{\kappa}}$, $\underline{\kappa}\in\mathcal{N}$.

\end{section}

\begin{section}{Second Cauchy problem in a weighted Banach space of Taylor series}\label{seccion2}

We provide the solution of a Cauchy problem with analogous equation as the one studied in the previous section, written as a formal power series in $z$ with coefficients in an appropriate Banach space of functions in the variable $\tau$ and the perturbation parameter $\epsilon$. In Section~\ref{seccion1}, the domain of holomorphy of the coefficients remains invariant from the domain of holomorphy of the initial conditions. This happens so because the dilation operator $\tau\mapsto q^{-1}\tau$ sends points in any infinite sector in the complex plane with vertex at the origin into itself. Now, the domain of holomorphy of the coefficients for the formal solution of the Cauchy problem under study depends on the index considered. More precisely, if the initial conditions present a singularity at some point $a\in\C$ in the variable $\tau$, the coefficients of the formal solution of the Cauchy problem have singularities in $\tau$ that tend to 0, providing a small divisor phenomenon.

For every $\rho>0$, $\dot{D}_{\rho}$ stands for the set $D(0,\rho)\setminus\{0\}$. We preserve the value of the positive constants $M,A_1,C$ and $\delta_1$ from the previous section. Let $r_0>0$ with $\overline{\mathcal{E}}\subseteq D(0,r_0)$ and $(\hat{R}_{\beta})_{\beta\ge0}$ be a sequence of positive real numbers.

\begin{defin}
Let $\beta\in\N$. For $r_0>0$ and $\epsilon\in D(0,r_0)\setminus\{0\}$, $E^{2}_{\beta,\epsilon,\dot{D}_{\hat{R}_{\beta}}}$ stands for the vector space of functions $v\in\mathcal{O}(\dot{D}_{\hat{R}_{\beta}})$ such that 
$$|v(\tau)|_{\beta,\epsilon,\dot{D}_{\hat{R}_{\beta}}}:=\sup_{\tau\in\dot{D}_{\hat{R}_{\beta}}}\left\{|v(\tau)|\frac{|\epsilon|^{C\beta}}{e^{M\log^{2}\left(|\tau|+\delta_1\right)}}\right\}q^{-A_{1}\beta^{2}}$$
is finite. Let $\delta>0$. We write $H_{2}(\epsilon,\delta)$ for the vector space of all formal power $v(\tau,z)=\sum_{\beta\ge0}v_{\beta}(\tau)z^{\beta}/\beta!$ such that $v_{\beta}\in E^{2}_{\beta,\epsilon,\dot{D}_{\hat{R}_{\beta}}}$ with 
$$|v(\tau,z)|_{(\epsilon,\delta)}:=\sum_{\beta\ge0}|v_{\beta}(\tau)|_{\beta,\epsilon,\dot{D}_{\hat{R}_{\beta}}}\frac{\delta^{\beta}}{\beta!}<\infty.$$
The pair $(H_{2}(\epsilon,\delta),|\cdot|_{(\epsilon,\delta)})$ is a Banach space.
\end{defin}

\begin{lemma}\label{lema1302}
Let $s,\ell_{0},\ell_{1},m_{1},m_{2}\in\N$, $\delta>0$ and $\epsilon\in D(0,r_0)\setminus\{0\}$. We assume that
\begin{equation}\label{e323}
C(\ell_1+s)-\ell_0\ge 0, \quad m_2-2A_1(\ell_1+s)>0.
\end{equation}
Moreover, we assume that the elements of the sequence $(\hat{R}_{\beta})_{\beta\ge0}$ are such that
\begin{equation}\label{e1302}
\hat{R}_{\beta}\le q^{m_{1}}\hat{R}_{\beta-\ell_{1}-s},
\end{equation}
for every $\beta\ge \ell_{1}+s$.

Under the previous assumptions, there exists a positive constant $C_{21}$ which depends on $C, q, m_1, m_2,s, \ell_0, \ell_1, M, A_1, \delta_1,  r_0$ (not depending on $\epsilon$ nor $\delta$), such that
$$\left|z^{s}\left(-\frac{\tau}{\epsilon}\right)^{\ell_0}\frac{1}{q^{m_{1}(\ell_0+1)}}(\partial_{z}^{-\ell_1}v)(\tau q^{-m_1},zq^{m_2})\right|_{(\epsilon,\delta)}\le C_{21}\delta^{\ell_{1}+s}\left|v(\tau,z)\right|_{(\epsilon,\delta)},$$
for every $v\in H_2(\epsilon,\delta)$.
\end{lemma}
\begin{proof}
Let $v(\tau,z)=\sum_{\beta\ge0}v_{\beta}(\tau)\frac{z^{\beta}}{\beta!}$ be an element of $H_{2}(\epsilon,\delta)$. We have 
\begin{multline}\left|z^{s}\left(-\frac{\tau}{\epsilon}\right)^{\ell_0}\frac{1}{q^{m_{1}(\ell_0+1)}}(\partial_{z}^{-\ell_1}v)(\tau q^{-m_1},zq^{m_2})\right|_{(\epsilon,\delta)}\\
=\left|\sum_{\beta\ge\ell_{1}+s}v_{\beta-\ell_{1}-s}(\tau q^{-m_{1}})q^{m_{2}(\beta-s)-m_{1}(\ell_{0}+1)}\frac{\beta!}{(\beta-s)!}\left(-\frac{\tau}{\epsilon}\right)^{\ell_{0}}\frac{z^{\beta}}{\beta!}\right|_{(\epsilon,\delta)}.\label{e1432}
\end{multline}
From (\ref{e1302}), one derives that for every $\tau\in \dot{D}_{\hat{R}_{\beta-\ell_1-s}}$, $v_{\beta-\ell_{1}-s}(\tau q^{-m_{1}})$ is well defined. In addition to this, the function $\tau\mapsto v_{\beta-\ell_{1}-s}(\tau q^{-m_{1}})$ is holomorphic in $\dot{D}_{\hat{R}_{\beta}}$ for every $\beta\ge \ell_{1}+s$. The expression in (\ref{e1432}) equals
\begin{equation}\label{e1472}
\sum_{\beta\ge\ell_{1}+s}\left|v_{\beta-\ell_{1}-s}(\tau q^{-m_{1}})q^{m_{2}(\beta-s)-m_{1}(\ell_{0}+1)}\frac{\beta!}{(\beta-s)!}\left(-\frac{\tau}{\epsilon}\right)^{\ell_{0}}\right|_{\beta,\epsilon,\dot{D}_{\hat{R}_{\beta}}}\frac{\delta^{\beta}}{\beta!}.
\end{equation}
Let $\beta\ge \ell_{1}+s$. From the definition of the norm $\left|\cdot\right|_{\beta,\epsilon,\dot{D}_{\hat{R}_{\beta}}}$, we get
$$\left|v_{\beta-\ell_{1}-s}(\tau q^{-m_{1}})q^{m_{2}(\beta-s)-m_{1}(\ell_{0}+1)}\frac{\beta!}{(\beta-s)!}\left(-\frac{\tau}{\epsilon}\right)^{\ell_{0}}\right|_{\beta,\epsilon,\dot{D}_{\hat{R}_{\beta}}}$$
\begin{align}
&=\sup_{\tau\in \dot{D}_{\hat{R}_{\beta}}}\left\{|v_{\beta-\ell_{1}-s}(\tau q^{-m_{1}})|\frac{|\epsilon|^{C(\beta-\ell_1-s)}}{e^{M\log^2\left(\frac{|\tau|}{q^{m_1}}+\delta_1\right)}}\right.\nonumber\\
&\left.\times\left|\frac{\tau}{\epsilon}\right|^{\ell_0}|\epsilon|^{C(\ell_1+s)}e^{M\left(\log^2\left(\frac{|\tau|}{q^{m_1}}+\delta_1\right)-\log^2(|\tau|+\delta_1)\right)}\right\}q^{p_1(\beta)}\frac{\beta!}{(\beta-s)!}q^{-A_{1}(\beta-\ell_{1}-s)^{2}},\label{e1592}
\end{align}
with $p_1(\beta)=m_{2}(\beta-s)-m_{1}(\ell_{0}+1)-A_{1}\beta^{2}+A_{1}(\beta-\ell_{1}-s)^{2}$.

The result follows provided that one is able to estimate the expression
$$q^{p_{1}(\beta)}\frac{\beta!}{(\beta-s)!}|\tau|^{\ell_0}|\epsilon|^{C(\ell_1+s)-\ell_0}e^{M\left(\log^2\left(\frac{|\tau|}{q^{m_1}}+\delta_1\right)-\log^2(|\tau|+\delta_1)\right)}.$$

From the first of the hypotheses made in (\ref{e323}), $|\epsilon|^{C(\ell_1+s)-\ell_0}$ is upper bounded by a constant. Also, taking into account (\ref{e1302}), there exists $\hat{R}>0$ such that $|\tau|\le \hat{R}$ for every $\tau\in\cup_{\beta\ge0} \dot{D}_{\hat{R}_{\beta}}$, so that 
$$|\tau|^{\ell_0}\exp\left(M(\log^2(|\tau|q^{-m_1}+\delta_1)-\log^2(|\tau|+\delta_1))\right)\le D_{21}(q,m_1,\delta_1,M,\ell_0),$$ 
for some positive constant $D_{21}$. The result immediately follows from (\ref{e323}) that guarantees that $\beta!/(\beta-s)!q^{p_{1}(\beta)}$ is bounded from above. 

\end{proof}

Let $\hat{R}>0$ be as in the proof of the previous lemma, i.e. $\hat{R}\ge \hat{R}_{\beta}$ for every $\beta\ge0$.

\begin{lemma}\label{lema2072}
Let $F(\epsilon,\tau)$ be a holomorphic and bounded function defined on $D(0,r_0)\times D(0,\hat{R})$.

Then, there exists a constant $C_{22}=C_{22}(F)>0$ such that
$$\left|F(\epsilon,\tau)v_{\epsilon}(\tau,z)\right|_{(\epsilon,\delta)}\le C_{22}\left|v_{\epsilon}(\tau,z)\right|_{(\epsilon,\delta)}$$
for every $\epsilon\in D(0,r_0)\setminus\{0\}$, every $\delta>0$ and all $v_{\epsilon}\in H_2(\epsilon,\delta)$.
\end{lemma}
\begin{proof}
Direct calculations on the definition of the norms in the space $H_2(\epsilon,\delta)$ allow us to conclude when taking $C_{22}:=\max\{|F(\epsilon,\tau)|:\epsilon\in D(0,r_0),\tau\in D(0,\hat{R})\}$.
\end{proof}

Let $S\ge1$, and $\mathcal{N}$ be a finite subset of $\N^{2}$. We also fix $a\in\C\setminus\R_{+}$ such that $|a|\ge\hat{R}$, with $\hat{R}$ as before.


Let $m_{\underline{\kappa},1}$, $m_{\underline{\kappa},2}$ and $b_{\underline{\kappa}}$ be as in Section~\ref{seccion1}, for every $\underline{\kappa}=(\kappa_{0},\kappa_{1})$. 

We consider the functional equation
\begin{equation}\label{e2142}
\partial_{z}^{S}W(\epsilon,\tau,z)=\sum_{\underline{\kappa}=(\kappa_{0},\kappa_{1})\in\mathcal{N}}\frac{b_{\underline{\kappa}}(\epsilon,z)}{(a-\tau)q^{m_{\underline{\kappa},1}(\kappa_{0}+1)}}\left(-\frac{\tau}{\epsilon}\right)^{\kappa_{0}}\left(\partial_{z}^{\kappa_{1}}W\right)\left(\epsilon,q^{-m_{\underline{\kappa},1}}\tau,q^{m_{\underline{\kappa},2}}z\right)
\end{equation}
with initial conditions
\begin{equation}\label{e2212}
\left(\partial_{z}^{j}W\right)(\epsilon,\tau,0)=W_{j}(\epsilon,\tau)\quad, 0\le j\le S-1,
\end{equation}
where the function $(\epsilon,\tau)\mapsto W_{j}(\epsilon,\tau)$ is an element in $\mathcal{O}((D(0,r_0)\setminus\{0\})\times\dot{D}_{\hat{R}_0})$ for every $0\le j\le S-1$. 

We make the following 

\textbf{Assumption (A')} For every $\underline{\kappa}=(\kappa_{0},\kappa_{1})\in\mathcal{N}$ and every $s\in I_{\underline{\kappa}}$, we assume
$$ C (S-\kappa_{1}+s)-\kappa_{0}\ge0, \quad  m_{\underline{\kappa},2}-2A_{1}(S-\kappa_{1}+s)>0.$$

\textbf{Remark:} Observe that Assumption (A) implies Assumption (A').

\textbf{Assumption (B')} We assume $\hat{R}_{\beta}\le q^{m_{\underline{\kappa},1}}\hat{R}_{\beta-\kappa_1-s}$ for every $\underline{\kappa}=(\kappa_0,\kappa_1)\in\mathcal{N}$, every $s\in I_{\underline{\kappa}}$ and every $\beta\ge \kappa_1+s$.

We first state a result which provides a concrete value for the elements in $(\hat{R}_{\beta})_{\beta\ge0}$ under Assumption (B'). The choice is made in two respects: first, to clarify how the singularities suffer propagation in the formal solution of (\ref{e2142})+(\ref{e2212}), with respect to the variable $\tau$; and second, to provide acceptable domains of holomorphy for such coefficients when regarding this phenomenon of propagation of singularities. Any other appropriate choice for the elements in $(\hat{R}_{\beta})_{\beta\ge0}$ regarding these issues would also be fairish for our purpose. 

\begin{lemma}\label{lema426}
Let $\hat{d}_{1},\hat{d}_{2}>0$ and $\epsilon\in D(0,r_0)\setminus\{0\}$. 

We put $\hat{R}_{\beta}:=\hat{R}_0$ for $\beta=0,1,...,S-1$, and $\hat{R}_{\beta}=\hat{d}_{1}q^{\hat{d}_{2}\beta}$ for every $\beta\ge S$. Let us assume that (\ref{e2142})+(\ref{e2212}) has a formal solution in $z$,  $W(\epsilon,\tau,z)=\sum_{\beta\ge0}W_{\beta}(\epsilon,\tau)\frac{z^{\beta}}{\beta!}$. Then, there exists $\hat{d}_{20}$ such that for every $\hat{d}_{2}\ge \hat{d}_{20}$, the function $\tau\mapsto W_{\beta}(\epsilon,\tau)$ belongs to $\mathcal{O}(\dot{D}_{\hat{R}_{\beta}})$ for every $\beta\ge S$ and all $\epsilon\in D(0,r_0)\setminus\{0\}$.
\end{lemma}
\begin{proof}
Let $W(\epsilon,\tau,z)$ be a formal power series in $z$ of the form $\sum_{\beta\ge0}W_{\beta}(\epsilon,\tau)\frac{z^{\beta}}{\beta!}$. One can plug the formal power series into equation (\ref{e2142}) to obtain the recursion formula in (\ref{e305}) for the coefficients $(W_{\beta})_{\beta\ge S}$.  From this recurrence, one derives the domain of holomorphy for $W_{h+S}$ in the variable $\tau$ depends on the domain of holomorphy on $\tau$ of $W_{h_2+\kappa_1}$ and also on $q^{-m_{\underline{\kappa},1}}$ for every $\underline{\kappa}=(\kappa_0,\kappa_1)\in\mathcal{N}$, every $0\le h_{2}\le h$ such that $h-h_{2}\in I_{\underline{\kappa}}$. 

The initial conditions $W_{0},...,W_{S-1}$ are holomorphic functions in $\dot{D}_{\hat{R}_0}$. 

\begin{lemma}\label{lema7}
For every $N\ge 1$ the coefficients
$W_{NS-(N-1)\kappa_{10}},...,W_{(N+1)S-N\kappa_{10}}$ turn out to be holomorphic functions in $\dot{D}_{q^{N\overline{m}_{\underline{\kappa},1}}\hat{R}_0}$, for $\kappa_{10}:=\max\{\kappa_1:(\kappa_0,\kappa_1)\in\mathcal{N}\}$ and $\overline{m}_{\underline{\kappa},1}:=\max\{m_{\underline{\kappa},1}:\underline{\kappa}\in\mathcal{N}\}$. 
\end{lemma}
\begin{proof}

We prove it by recurrence on $N$ and regarding the recursion formula (\ref{e305}).

Let $N=1$. One has $h_2+\kappa_1\le S-1$ for any $h_2,\kappa_1$ as in (\ref{e305}) if and only if $h_2\le S-1-\kappa_1$ for every $(\kappa_0,\kappa_1)\in\mathcal{N}$, it is to say, if and only if $h_2\le S-1-\kappa_{10}$. In this case, $W_{h+S}$ only depends on the initial conditions $(W_{j})_{0\le j\le S-1}$. Moreover, $$h+S\in\{S,S+1,...,2S-\kappa_{10}-1\},$$
and the dilation on the variable $\tau$ allow us to obtain that $W_{S}$,...,$W_{2S-\kappa_{10}-1}$ are holomorphic functions in $\dot{D}_{q^{\overline{m}_{\underline{\kappa},1}}\hat{R}_0}$. 

The proof can be followed recursively for every $N\ge2$ by considering analogous blocks of indices as before.
\end{proof}

Regarding Lemma~\ref{lema7}, the proof of Lemma~\ref{lema426} is concluded if one can check that for every $N\ge1$, $\hat{R}_{\beta}\le \hat{R}_0q^{N\overline{m}_{\underline{\kappa},1}}$ whenever $$\beta\in\{NS-(N-1)\kappa_{10},...,(N+1)S-N\kappa_{10}-1\}=\{N(S-\kappa_{10})+\kappa_{10},...,N(S-\kappa_{10})+S-1\}.$$ 

Let $N\ge 1$ and $\beta=N(S-\kappa_{10})+L$, with $\kappa_{10}\le L\le S-1$. Let $\hat{d}_{1}\le \hat{R}_0$. We have $\hat{R}_{\beta}=\tilde{d}_{1}q^{\tilde{d}_{2}[N(S-\kappa_{10})+L]}\le \hat{R}_0q^{N\overline{m}_{\underline{\kappa},1}}$ if and only if $N\overline{m}_{\underline{\kappa},1}\le\tilde{d}_{2}[N(S-\kappa_{10})+L]$. The result follows for any $\tilde{d}_{2}\ge\frac{\overline{m}_{\underline{\kappa},1}}{S-\kappa_{10}}$.
\end{proof}

\begin{lemma}
Let $\hat{R}_{\beta}$ be defined as in Lemma~\ref{lema426}. Then, $(\hat{R}_{\beta})_{\beta\ge0}$ satisfies Assumption (B').
\end{lemma}
\begin{proof}
From the definition of $\hat{R}_{\beta}$, the lemma follows when taking $\tilde{d}_{2}\ge\frac{m_{\underline{\kappa},1}}{\kappa_{1}+s}$ for every $\underline{\kappa}=(\kappa_{0},\kappa_{1})\in\mathcal{N}$, and every $s\in I_{\underline{\kappa}}$.
\end{proof}

\textbf{Assumption (B''):} We assume $\hat{R}_{\beta}=\hat{R}_{0}$ for $0\le\beta\le S-1$ and $\hat{R}_{\beta}:=\hat{d}_{1}q^{\hat{d}_{2}\beta}$ for any $\hat{d}_{2}\ge \hat{d}_{20}$, with $\hat{d}_{20}>0$ as in Lemma~\ref{lema426}.

As it has been pointed out before, the Assumption (B') is substituted in the present work by Assumption (B'') with the cost of losing some generality, but giving concrete values for $\hat{R}_{\beta}$, for every $\beta\ge0$. The incoming theorem is valid when considering any other choice of the elements in $(\hat{R}_{\beta})_{\beta\ge0}$ satisfying Assumption (B').
 
\begin{theo}\label{teorema443}
Let Assumption (A') and Assumption (B'') be fulfilled. We also make the next assumption on the initial conditions (\ref{e2212}): there exist $\Delta>0$ and $0<\tilde{M}<M$ such that 
\begin{equation}\label{e2392}
|W_{j}(\epsilon,\tau)|\le\Delta e^{\tilde{M}\log^{2}(|\tau|+\delta_1)}|\epsilon|^{K_0},
\end{equation} 
for every $\tau\in D(0,\hat{R}_0)$, $\epsilon\in D(0,r_0)\setminus\{0\}$ and $0\le j \le S-1$, where $K_0=\max\{\kappa_0: (\kappa_0,\kappa_1)\in\mathcal{N}\}$. Then, there exists a formal power series $W(\epsilon,\tau,z)=\sum_{\beta\ge0}W_{\beta}(\epsilon,\tau)\frac{z^{\beta}}{\beta!}$, with $W_{\beta}(\epsilon,\tau)\in\mathcal{O}( (D(0,r_0)\setminus\{0\})\times\dot{D}_{\hat{R}_{\beta}})$, which provides a formal solution of (\ref{e2142})+(\ref{e2212}). Moreover, there exist positive constants $C_{23}$ and $C_{24}$ (only depending on $r_0$, $\hat{R}_0$, $q$, $C$, $S$, $A_1$, $\delta_1$, $M$), and $0<\delta<1$ such that
\begin{equation}\label{e469}\left|W_{\beta}(\epsilon,\tau)\right|\le C_{23}\beta!\left(\frac{C_{24}}{\delta}\right)^{\beta}|\epsilon|^{-C\beta}e^{M\log^{2}(|\tau|+\delta_1)}q^{A_{1}\beta^2},
\end{equation}
for every $\epsilon\in D(0,r_0)\setminus\{0\}$, $\tau\in\dot{D}_{\hat{R}_{\beta}}$, and for every $\beta\in\N$. 
\end{theo}

\begin{proof}
The proof follows analogous steps as the one of Theorem~\ref{teorema240}, so we do not enter into details not to repeat arguments.

Let $\epsilon\in D(0,r_0)\setminus\{0\}$ and $0<\delta<1$. The set $\mathbb{E}$ is taken to be $\{\mathcal{O}(\dot{D}_{\hat{R}_{\beta}}):\beta\ge0\}$. We consider the map $\mathcal{A}_{\epsilon}$ from $\mathbb{E}[[z]]$ into itself defined in the same way as in (\ref{e243}).

From Lemma~\ref{lema426} and Assumption (B''), the unique formal solution of (\ref{e2142})+(\ref{e2212}), determined by the recursion formula (\ref{e305}), $W(\epsilon,\tau,z)=\sum_{\beta\ge0}W_{\beta}(\epsilon,\tau)\frac{z^{\beta}}{\beta!}$, is such that $W_{\beta}(\epsilon,\tau)\in\mathcal{O}(\dot{D}(0,\hat{R}_{\beta}))$ for every $\beta\ge0$.  

Regarding the initial conditions of the Cauchy problem, one can reduce $\hat{d}_{1}$, if necessary, so that $\hat{R}_{j+s}q^{-m_{\underline{\kappa},1}}\le \hat{R}_{0}$ and so the map $\tau\mapsto W_{j+\kappa_{1}}(\epsilon,q^{-m_{\underline{\kappa},1}}\tau)$ is well defined, for every $\underline{\kappa}=(\kappa_{0},\kappa_{1})\in\mathcal{N}$, every $s\in I_{\underline{\kappa}}$ and $j=0,1,...,S-1-\kappa_{1}$. Moreover, from (\ref{e2392}), the expression
$$\left|\frac{z^{s}}{q^{m_{\underline{\kappa}},1}(\kappa_{0}+1)}\left(-\frac{\tau}{\epsilon}\right)^{\kappa_0}\partial_{z}^{\kappa_{1}}w_{\epsilon}(q^{-m_{\underline{\kappa},1}}\tau,q^{m_{\underline{\kappa},2}}z)\right|_{(\epsilon,\delta)}$$
can be estimated in an analogous manner as in the corresponding step of the proof of Theorem~\ref{teorema240}, for every $\underline{\kappa}=(\kappa_0,\kappa_1)\in\mathcal{N}$ and all $s\in I_{\underline{\kappa}}$.
\end{proof}


\end{section}

\begin{section}{Analytic solutions in a parameter of singularly perturbed Cauchy problem}\label{seccion3}

\subsection{Laplace transform and $q-$Gevrey asymptotic expansion}\label{subseccion31}

In this subsection, we recall some identities for the Laplace transform, and state some definitions and first results on $q-$Gevrey asymptotic expansions. The next lemma can be found in~\cite{malek3}.

\begin{lemma}
Let $m\in\N$, and $w_1(\tau)$ be a holomorphic function in an unbounded sector $U$ such that there exist $C,K>0$ with
\begin{equation}\label{e516}
|w(\tau)|\le C\exp(K|\tau|),
\end{equation}
for every $\tau \in U$. Let $\mathcal{D}$ be an unbounded sector with vertex at 0 which veryfies that 
$$d+\arg(t)\in(-\frac{\pi}{2},\frac{\pi}{2}),\quad \cos(d+\arg(t))\ge\delta_{2},$$
for some $d\in\R$ and $\delta_{2}>0$. Then, 
$$t\mapsto\int_{L_{d}}w(\tau)e^{-t\tau}d\tau$$
is a holomorphic and bounded function defined for $t\in\mathcal{D}\cap\{|t|>K/\delta_{2}\}$. Moreover, the following identities hold:
\begin{equation}\label{e524}
\int_{L_{d}}\tau^{m}e^{-t\tau}d\tau=\frac{m!}{t^{m+1}},\quad \partial_{t}\left(\int_{L_{d}}w(\tau)e^{-t\tau}d\tau\right)=\int_{L_{d}}(-\tau)w(\tau)e^{-t\tau}d\tau,
\end{equation}
where $L_{d}=\R_{+}e^{id}\subseteq U\cup\{0\}$, for all $t\in\mathcal{D}\cap\{|t|>K/\delta_{2}\}$.
\end{lemma}

In the sequel, we work with functions which satisfy more restrictive bounds that the ones in (\ref{e516}). Indeed, we deal with bounds of the form $C\exp(K\log^2|\tau|)$, for some $C,K>0$. This alters the asymptotic behaviour of the Laplace transform and cause the appearance of $q-$Gevrey asymptotic expansions, associated to estimates related to the sequence $(q^{-n^2})_{n\ge0}$. 

For any open sector $S=\{z\in\C:a<\arg(z)<b,|z|<\rho\}$ in the complex plane with vertex at 0 with $\rho$ fintite or infinite, and $0\le a<b\le 2\pi$, we say the finite sector $\tilde{S}$ with vertex at the origin is a proper subsector of $S$, and we denote it $\tilde{S}\prec S$, if $\tilde{S}=\{\tau\in\C^{\star}:a_1<\arg(\tau)<b_1,|z|<\tilde{\rho}\}$ for some $0\le a<a_1<b_1<b\le2\pi$, and some $\tilde{\rho}\in\R$, $0<\tilde{\rho}<\rho$.

$\mathbb{H}$ stands for a complex Banach space.

We preserve the Definition of $q-$Gevrey asymptotic expansion established in~\cite{lastramalek}, in order to be coherent with the definitions in that work.

\begin{defin}
Let $S$ be a sector in $\C^{\star}$ with vertex at the origin, and $A>0$. We say a holomorphic function $f:S\to\mathbb{H}$ admits the formal power series $\hat{f}=\sum_{n\ge0}f_n\epsilon^n\in\mathbb{H}[[\epsilon]]$ as its $q-$Gevrey asymptotic expansion of type $A$ in $S$ if for every $\tilde{S}\prec S$ there exist $C_1,H>0$ such that
\begin{equation}\label{e497000}
\left\|f(\epsilon)-\sum_{n=0}^{N}f_{n}\epsilon^n\right\|_{\mathbb{H}}\le C_{1}H^{N}q^{-A\frac{N^2}{2}}\frac{|\epsilon|^{N+1}}{(N+1)!},\quad N\ge0,
\end{equation}
for every $\epsilon\in\tilde{S}$.
\end{defin} 
 
The next proposition, detailed in~\cite{lastramalek} in the more general geometry of $q-$spirals, characterises null $q-$Gevrey asymptotic expansion.
\begin{prop}\label{prop587}
Let $A>0$ and $f:S\to\mathbb{H}$ a holomorphic function in a sector $S$ with vertex at the origin. Then,
\begin{enumerate}
\item[$i)$] If $f$ admits the power series with null coefficients, which is denoted by $\hat{0}$, as its $q-$Gevrey asymptotic expansion of type $A$, then for every $\tilde{S}\prec S$ there exists $C_1>0$ with
$$\left\|f(\epsilon)\right\|_{\mathbb{H}}\le C_1e^{-\frac{1}{\tilde{a}}\frac{1}{2(-\log(q))}\log^{2}|\epsilon|},$$
for every $\epsilon\in\tilde{S}$ and every $\tilde{a}>A$.
\item[$ii)$] If for every $\tilde{S}\prec S$ there exists $C_1>0$ with
$$\left\|f(\epsilon)\right\|_{\mathbb{H}}\le C_{1}e^{-\frac{1}{A}\frac{1}{2(-\log(q))}\log^{2}|\epsilon|},$$
for every $\epsilon\in \tilde{S}$ then $f$ admits $\hat{0}$ as its $q-$Gevrey asymptotic expansion of type $\tilde{a}$ in $S$, for every $\tilde{a}>A$.
\end{enumerate}
\end{prop}

%

\subsection{Analytic solutions in a parameter of singularly perturbed Cauchy problem}\label{subseccion32}

We recall the definition of a good covering.
\begin{defin}\label{goodcovering}
Let $\{\mathcal{E}_{i}\}_{0\le i\le \nu-1}$ be a finite family of open sectors with vertex at the origin and finite radius $\epsilon_0$. We assume that $\mathcal{E}_{i}\cap\mathcal{E}_{i+1}\neq\emptyset$ for $0\le i\le \nu-1$ (we put $\mathcal{E}_{\nu}:=\mathcal{E}_{0}$), and also that $D(0,\nu_0)\setminus\{0\}\subseteq\cup_{i=0}^{\nu-1}\mathcal{E}_{i}$ for some $\nu_0>0$. Then, the family $\{\mathcal{E}_{i}\}_{1\le i\le\nu-1}$ is known as a good covering in $\C^{\star}$. 
\end{defin}

\begin{defin}
Let $\{\mathcal{E}_{i}\}_{0\le i\le \nu-1}$ be a good covering in $\C^{\star}$. We consider a family $\{\{\tilde{S}_{i}\}_{0\le i\le\nu-1},\mathcal{T}\}$ such that:
\begin{enumerate}
\item There exist $d_i\in[0,2\pi)$, $0<\theta_{i}<\frac{\pi}{2}$ such that
$$\tilde{S}_{i}=\tilde{S}_{i}(d_{i},\theta_{i}):=\left\{t\in\C^{\star}:|\arg(t)-d_{i}|<\frac{\theta_i}{2}\right\},$$ 
for every $0\le i\le \nu-1$.
\item $\mathcal{T}$ is an unbounded subset of an open sector with vertex at the origin. We assume $|t|\ge r_{\mathcal{T}}$ for every $t\in\mathcal{T}$.
\item For every $0\le i\le \nu-1$ and $\tau\in\tilde{S}_{i}$, there exists $\delta_3>0$ such that $|a-\tau|>\delta_3$.
\item For every $0\le i \le \nu-1$, $t\in\mathcal{T}$ and $\epsilon\in \mathcal{E}_{i}$, one has $t/\epsilon \in \tilde{S}_{i}$.
\end{enumerate}
Under the previous settings, we say the family $\{\{\tilde{S}_{i}\}_{0\le i\le \nu-1},\mathcal{T}\}$ is associated to the good covering $\{\mathcal{E}_{i}\}_{0\le i\le \nu-1}$.
\end{defin}


Let us consider a good covering in $\C^{\star}$, $\{\mathcal{E}_{i}\}_{0\le i\le \nu-1}$.

Let $S\ge 1$ and $a\in\C\setminus\R_{+}$. We consider a finite subset of $\N^2$, $\mathcal{N}$. For every $\underline{\kappa}=(\kappa_0,\kappa_1)\in\mathcal{N}$, let $m_{\underline{\kappa},1},m_{\underline{\kappa},2}\in\N$, and  $b_{\underline{\kappa}}(\epsilon,z)$ a holomorphic and bounded function on $D(0,r_0)\times \C$, for some $r_0>0$. For each $0\le i \le \nu-1$, we consider the main Cauchy problem in the present work:

\begin{equation}\label{e497}
\epsilon \partial_{t}\partial_{z}^{S}X_i(\epsilon,t,z)+a\partial_{z}^{S}X_i(\epsilon,t,z)=\sum_{\underline{\kappa}=(\kappa_{0},\kappa_{1})\in\mathcal{N}}b_{\underline{\kappa}}(\epsilon,z)(\partial_{t}^{\kappa_{0}}\partial_{z}^{\kappa_{1}}X_i)(\epsilon,q^{m_{\underline{\kappa},1}}t,q^{m_{\underline{\kappa},2}}z),
\end{equation}

with initial conditions
\begin{equation}\label{e502}
(\partial_{z}^{j}X_{i})(\epsilon,t,0)=\phi_{i,j}(\epsilon,t)\quad 0\le j\le S-1,
\end{equation}

where the functions $\phi_{i,j}(\epsilon,t)$ are constructed as follows. Let $\{\{\tilde{S}_{i}\}_{0\le i\le \nu-1},\mathcal{T}\}$ be a family of open sets associated to the good covering $\{\mathcal{E}_{i}\}_{0\le i\le\nu-1}$. 

From now on, we assume the values of $(R_{\beta})_{\beta\ge0}$ and $(\hat{R}_{\beta})_{\beta\ge0}$ are those in the preceeding sections. If necessary, one can adjust the values of $d_1,d_2,\hat{d}_{1}$ and $\hat{d}_{2}$ so that $R_{\beta}<\hat{R}_{\beta}$ for every $\beta\ge0$ so that $\dot{D}_{\hat{R}_{\beta}}\cap \tilde{S}_{i,\beta}\neq\emptyset$ for every $\beta\ge0$ and every $0\le i\le \nu-1$. Here, we have put
$$\tilde{S}_{i,\beta}:=\left\{\tau\in\C^{\star}: \tau\in \tilde{S}_{i}, |\tau|>R_{\beta}\right\},\quad \beta\ge0, 0\le i\le \nu-1.$$

For every $0\le j\le \nu-1$, we assume that $(\epsilon,\tau)\mapsto W_j(\epsilon,\tau)$ is a bounded and holomorphic function on $(D(0,r_0)\setminus\{0\})\times \dot{D}_{\hat{R}_{j}}$ verifying 
$$|W_j(\epsilon,\tau)|\le \Delta e^{\tilde{M}\log^{2}(\frac{|\tau|}{|\epsilon|}+\delta_1)}|\epsilon|^{K_0},$$
for every $(\epsilon,\tau)\in(D(0,r_0)\setminus\{0\})\times \dot{D}_{\hat{R}_{j}}$. Here $\tilde{M},K_0,\Delta,\delta_1$ are the constants provided in Theorem~\ref{teorema240}.  Assume moreover that $W_j(\epsilon,\tau)$ can be extended to an analytic function $(\epsilon,\tau)\mapsto W_{\mathcal{E}_{i},S_{j},j}(\epsilon,\tau)$ defined on $\mathcal{E}_{i}\times \tilde{S}_{i,j}$ and
$$|W_{\mathcal{E}_{i},\tilde{S}_{i,j},j}(\epsilon,\tau)|\le \Delta e^{\tilde{M}\log^{2}(|\tau|+\delta_1)}|\epsilon|^{K_0},$$
for every $(\epsilon,\tau)\in \mathcal{E}_{i}\times \tilde{S}_{i,j}$.

Take $\gamma_{i}$ such that $\R_{+}e^{\gamma_{i}\sqrt{-1}}\subseteq D_{\hat{R}_{j}}\cup \tilde{S}_{i,j}$. We put 
$$\phi_{i,j}(t,\epsilon):=\int_{L_{\gamma_{i}}}W_{\mathcal{E}_{i},S_{j},j}(\epsilon,\tau)e^{-\frac{t\tau}{\epsilon}}d\tau,$$
for every $(\epsilon,t)\in \mathcal{E}_{i}\times \mathcal{T}$. One can check that $\phi_{i,j}$ is well defined and holomorphic in $\mathcal{T}\times \mathcal{E}_{i}$. Indeed, there exists $\delta_2>0$ such that $\cos(\gamma_{j}-\arg(t/\epsilon))>\delta_2$ for every $(t,\epsilon)\in \mathcal{E}_{i}\times \mathcal{T}$. Moreover, from the growth properties of $W_{\mathcal{E}_{i},\tilde{S}_{i,j},j}(\epsilon,\tau)$, one deduces
$$\left|\int_{L_{\gamma_{i}}}W_{\mathcal{E}_{i},S_{j},j}(\epsilon,\tau)e^{-\frac{t\tau}{\epsilon}}d\tau\right|\le\int_{0}^{\infty}|W_{\mathcal{E}_{i},S_{j},j}(\epsilon,se^{\sqrt{-1}\gamma_{i}})|e^{-\frac{|t|\delta_{2}s}{|\epsilon |}}ds\le  \Delta|\epsilon|^{K_0}\int_{0}^{\infty}e^{\tilde{M}\log^{2}(\frac{s}{|\epsilon|}+\delta_1)}e^{-\frac{s\delta_{2}r_{\mathcal{T}}}{|\epsilon|}}ds,$$
which is convergent for every $(\epsilon,t)\in\mathcal{E}_{i}\times\mathcal{T}$.

\begin{theo}\label{teorema549}
Let Assumptions (A), (B) and (B'') be fulfilled. For every $0\le i \le \nu-1$, we consider the problem (\ref{e497})+(\ref{e502}) with initial conditions constructed as above. Then, the problem (\ref{e497})+(\ref{e502}) admits a solution $X_{i}(\epsilon,t,z)$ which is holomorphic and bounded in $\mathcal{E}_{i}\times\mathcal{T}\times\C$.

Moreover, for every $0\le i\le \nu-1$ and for every $\Delta>1$ there exists $E_1>0$ (not depending on $\epsilon$), such that
\begin{equation}\label{e552}
\sup_{\stackrel{t\in\mathcal{T}}{z\in D(0,\rho)}}|X_{i+1}(\epsilon,t,z)-X_{i}(\epsilon,t,z)|\le E_{1}e^{-\frac{A_1}{2d_2^2\Delta}\frac{1}{(-\log(q))}\log^2|\epsilon|},
\end{equation}
for every $\epsilon\in\mathcal{E}_{i}\cap\mathcal{E}_{i+1}$ (where, by convention, $X_{\nu}:=X_{0})$.
\end{theo}
\begin{proof}
Let $0\le i\le \nu-1$ and $\epsilon\in\mathcal{E}_{i}$. We consider the Cauchy problem (\ref{e214}) with initial conditions given by 
\begin{equation}\label{e550}
(\partial_{z}^{j}W)(\epsilon,\tau,0)=W_{\mathcal{E}_{i},\tilde{S}_{i,j},j}(\epsilon,\tau),\quad 0\le j\le S-1.
\end{equation}
Theorem~\ref{teorema240} shows that the problem (\ref{e239})+(\ref{e550}) has a formal solution $W(\epsilon,\tau,z)=\sum_{\beta\ge0}W_{\beta}(\epsilon,\tau)\frac{z^{\beta}}{\beta!}$, with $W_{\beta}\in\mathcal{O}(\mathcal{E}_{i}\times \tilde{S}_{i,\beta})$ for every $\beta\ge0$. Moreover, for every $\beta\ge0$ one has
\begin{equation}\label{e560}
|W_{\beta}(\epsilon,\tau)|\le C_{13}\beta!\left(\frac{C_{14}}{\delta}\right)^{\beta}e^{M\log^2(\frac{|\tau|}{|\epsilon|}+\delta_1)}\left|\frac{\tau}{\epsilon}\right|^{C\beta}q^{A_1\beta^{2}},
\end{equation}
for every $(\epsilon,\tau)\in\mathcal{E}_{i}\times \tilde{S}_{i,\beta}$, where $C_{13},C_{14}$ and $\delta$ are positive constants provided in the proof of Theorem~\ref{teorema240}.
In a paralel direction, one can consider the same Cauchy problem with initial conditions given by
\begin{equation}\label{e557}
(\partial_{z}^{j}W)(\epsilon,\tau,0)=W_{j}(\epsilon,\tau),\quad 0\le j\le S-1,
\end{equation}
where $W_{j}\in\mathcal{O}((D(0,r_{0})\setminus\{0\})\times \dot{D}_{\hat{R}_{j}})$ are as above.

From Theorem~\ref{teorema443}, one concludes that the formal power series $W(\epsilon,\tau,z)$ is such that $W_{\beta}$ can be extended to a holomorphic function defined in $(D(0,r_0)\setminus\{0\})\times \dot{D}_{\hat{R}_{\beta}} $, for every $\beta\ge0$. We preserve notation for these extensions. Moreover, for every $\beta\ge0$ one has
\begin{equation}\label{e561}
|W_{\beta}(\epsilon,\tau)|\le C_{23}\beta!\left(\frac{C_{24}}{\delta}\right)^{\beta}|\epsilon|^{-C\beta}e^{M\log^{2}(|\tau|+\delta_1)}q^{A_{1}\beta^2},
\end{equation}
for every $(\epsilon,\tau)\in (D(0,r_0)\setminus\{0\})\times \dot{D}_{\hat{R}_{\beta}} $, and some positive constants $C_{23}$ and $C_{24}$ determined in the proof of Theorem~\ref{teorema443}.

We put $X_{i}(\epsilon,\tau,z)=\sum_{\beta\ge0}X_{i,\beta}(t,\epsilon)\frac{z^{\beta}}{\beta!}$, where 
$$X_{i,\beta}(\epsilon,t):=\int_{L_{\gamma_{i}}}W_{\beta}(\tau,\epsilon)e^{-\frac{\tau t}{\epsilon }}d\tau.$$

We fist check that $X_{i}$ is, at least formally, a solution of (\ref{e497})+(\ref{e502}). From (\ref{e524}), one can check by inserting the formal power series $X_{i}$ in (\ref{e497}), that it turns out to be a formal solution in the variable $z$ of (\ref{e497})+(\ref{e502}) if and only if $W(\epsilon,\tau,z)$ is a formal solution of (\ref{e214})+(\ref{e221}) and (\ref{e2142})+(\ref{e2212}).   

Bearing in mind that $W_{\beta}$ verifies (\ref{e557}) and (\ref{e561}), one derives $X_{i,\beta}$ is well defined in $\mathcal{E}_{i}\times\mathcal{T}$, for every $\beta\ge0$. We now state a proof for the fact that $(\epsilon,\tau,z)\mapsto X_{i}(\epsilon,\tau,z)$ is indeed a holomorphic solution of (\ref{e497})+(\ref{e502}) in $\mathcal{E}_{i}\times\mathcal{T}\times\C$. Let $\epsilon\in\mathcal{E}_{i}$, $t\in\mathcal{T}$, and $\b\in\N$. One has
\begin{equation}\label{e577}
\left|\int_{L_{\gamma_{i}}}W_{\beta}(\tau,\epsilon)e^{-\frac{t\tau}{\epsilon }}d\tau\right|\le \left|\int_{L_{\gamma_{i},1}}W_{\beta}(\tau,\epsilon)e^{-\frac{t\tau}{\epsilon }}d\tau\right|+\left|\int_{L_{\gamma_{i},2}}W_{\beta}(\tau,\epsilon)e^{-\frac{t\tau}{\epsilon }}d\tau\right|+\left|\int_{L_{\gamma_{i},3}}W_{\beta}(\tau,\epsilon)e^{-\frac{t\tau}{\epsilon }}d\tau\right|,
\end{equation}   
where $L_{\gamma_{i},1}:=L_{\gamma_{i}}\cap D_{\hat{R}_{\beta}}$, $L_{\gamma_{i},2}:=L_{\gamma_{i}}\cap \tilde{S}_{i,\beta}$ and $L_{\gamma_{i},3}:=L_{\gamma_{i}}\cap \dot{D}_{\hat{R}_{\beta}}\cap \tilde{S}_{i,\beta} $. 
We only give details on the first and second integrals appearing on the right-hand side of the previous inequality. The first integral on the right-hand side of (\ref{e577}) can be upper bounded by means of (\ref{e561}), and the choice of direction $\gamma_{i}$.
\begin{align}
\left|\int_{L_{\gamma_{i},1}}W_{\beta}(\tau,\epsilon)e^{-\frac{\tau}{\epsilon t}}d\tau\right|&\le \int_{0}^{\hat{R}_{\beta}}|W_{\beta}(se^{\sqrt{-1}\gamma_{i}},\epsilon)|e^{-\frac{s|t|}{|\epsilon |}\cos(\gamma_{i}-\arg(t/\epsilon))}ds\nonumber\\
&\le C_{23}\beta!\left(\frac{C_{24}}{\delta}\right)^{\beta}q^{A_{1}\beta^2}\int_{0}^{\hat{R}_{\beta}}|\epsilon|^{-C\beta}e^{M\log^{2}(\frac{s}{|\epsilon|}+\delta_{1})}e^{-\frac{s\delta_{2}r_{\mathcal{T}}}{|\epsilon|}}ds.\label{e582}
\end{align}

One has 
$$|\epsilon|^{-C\beta}e^{M\log^2\left(\frac{s}{|\epsilon|}+\delta_1\right)}e^{-\frac{s\delta_2r_{\mathcal{T}}}{|\epsilon|}}\le\tilde{C}_{23}|\epsilon|^{-C\beta}e^{-\frac{s\delta_2r_{\mathcal{T}}}{2|\epsilon|}},$$
for some $\tilde{C}_{23}>0$. Now, the function $x\mapsto x^{-C\beta}e^{-\frac{s\delta_2r_{\mathcal{T}}}{|\epsilon|}}$ attains its maximum at $x=\frac{C\beta 2}{s\delta_{2}r_{\mathcal{T}}}$. One can reduce $r_0$, if neccesary, to conclude that this function is increasing for $x\in[0,\epsilon_0]$. The expression in (\ref{e582}) is upper bounded by
$$C_{23}\beta!\left(\frac{C_{24}}{\delta r_0^{C}}\right)^{\beta}q^{A_{1}\beta^{2}}\int_{0}^{\infty}e^{-\frac{s\delta_{2}r_{\mathcal{T}}}{2r_0}}ds.$$
This yields
\begin{equation}\label{e588}
\left|\int_{L_{\gamma_{i},1}}W_{\beta}(\tau,\epsilon)e^{-\frac{t\tau}{\epsilon }}d\tau\right|\le C_{31}\beta!C_{32}^{\beta}q^{A_{1}\beta^{2}},
\end{equation}
for some constants $C_{31},C_{32}>0$ only depending on $\delta$, $r_0$, $R_0$, $q$, $S$, $A_1$ , $C$, $\delta_{1}$, $M$, $r_{\mathcal{T}}$, $\delta_{2}$.
We now consider the second integral appearing on the right-hand side of (\ref{e577}). From (\ref{e560}) and similar estimates as before we get
\begin{equation}\label{alpha1}
\left|\int_{L_{\gamma_{i},2}}W_{\beta}(\tau,\epsilon)e^{-\frac{t\tau}{\epsilon }}d\tau\right|\le C_{13}\beta!\left(\frac{C_{14}}{\delta}\right)^{\beta}q^{A_{1}\beta^{2}}\int_{R_{\beta}}^{\infty}e^{M\log^{2}(\frac{s}{|\epsilon|}+\delta_{1})}\left(\frac{s}{|\epsilon|}\right)^{C\beta}e^{-\frac{s\delta_{2}r_{\mathcal{T}}}{|\epsilon|}}ds.
\end{equation}

The function $x\mapsto g_1(x)= e^{M\log^2(x+\delta_1)}x^{C\beta}$, $x\ge0$ is such that $g_1(x)\le g_2(x)$ for all $x\ge0$, where $g_{2}(x)=\tilde{C}_{13}e^{M\log^2(x)}x^{C\beta}$, for some positive constant $\tilde{C}_{13}$, not depending on $\beta$. $g_2$ attains its maximum value at $x_0=\exp(-\frac{C\beta}{2M})$ so that $g(x)\le g(x_0)=\exp(-\frac{C^2\beta^2}{4M}),$ for every $x>0$. This implies 
$$e^{M\log^2\left(\frac{s}{|\epsilon|}+\delta_1\right)}\left(\frac{s}{|\epsilon|}\right)^{C\beta}e^{-\frac{s\delta_2r_{\mathcal{T}}}{|\epsilon|}}\le e^{-\frac{C^{2}}{4M}\beta^2}e^{-\frac{s\delta_2 r_{\mathcal{T}}}{|\epsilon|}}.$$
From (\ref{alpha1}) we derive
\begin{equation}\label{e600}
\left|\int_{L_{\gamma_{i},2}}W_{\beta}(\tau,\epsilon)e^{-\frac{t\tau}{\epsilon }}d\tau\right|\le C_{13}\int_{0}^{\infty}e^{-\frac{s\delta_2r_{\mathcal{T}}}{2r_0}}ds\beta!q^{A_{1}\beta^2}e^{-\frac{C^{2}}{4M}\beta^2}e^{-\frac{R_{\beta\delta_2 r_{\mathcal{T}}}}{2r_0}}=\tilde{C}_{14}\beta!q^{\left(A_{1}-\frac{C^2}{4M\log(q)}\right)\beta^2},
\end{equation}
for some $\tilde{C}_{13}>0$.

From (\ref{e588}) and (\ref{e600}), we lead to the existence of positive constants $C_{41},C_{42}$, not depending on $\beta$, such that 
$$\left|\sum_{\beta\ge0}X_{i,\beta}(t,\epsilon)\frac{z^{\beta}}{\beta!}\right|\le C_{41}\sum_{\beta\ge0}C_{42}^{\beta}q^{A_{1}\beta^{2}}|z|^{\beta},$$
for every $z\in\C$. This allows us to conclude the first part of the proof.

Let $0\le i\le \nu-1$ and $\rho>0$. For every $(\epsilon,t,z)\in (\mathcal{E}_{i}\cap\mathcal{E}_{i+1})\times \mathcal{T}\times D(0,\rho)$ we have
$$|X_{i+1}(\epsilon,t,z)-X_{i}(\epsilon,t,z)|\le \sum_{\beta\ge0}|X_{i+1,\beta}(\epsilon,t)-X_{i,\beta}(\epsilon,t)|\frac{\rho^{\beta}}{\beta!}.$$
We can write
\begin{align*}
X_{i+1,\beta}(\epsilon,t)-X_{i,\beta}(\epsilon,t)&=\int_{L_{\gamma_{i+1},2}} W_{\beta}(\tau,\epsilon)e^{-\frac{t\tau}{\epsilon }}d\tau- \int_{L_{\gamma_{i},2}} W_{\beta}(\tau,\epsilon)e^{-\frac{t\tau}{\epsilon }}d\tau \\
&+\int_{L_{\gamma_{i+1},4}-L_{\gamma_{i},4} } W_{\beta}(\tau,\epsilon)e^{-\frac{t\tau}{\epsilon }}d\tau,
\end{align*}
where $L_{\gamma_{i+1},4}-L_{\gamma_{i},4}$ stands for the path consisting of two parts: the first one going from $R_{\beta}e^{\sqrt{-1}\gamma_{i+1}}$ to 0 along the segment $[0,R_{\beta}e^{\sqrt{-1}\gamma_{i+1}}]$ and the path going from 0 to $R_{\beta}e^{\sqrt{-1}\gamma_{i}}$ following direction $\gamma_{i}$.

This integral has already been estimated in (\ref{e600}), for the first part of the proof, so we omit the details. We also omit the details on the integral concerning the path $L_{\gamma_{i+1},2}$ which is analogous. 

In order to estimate the integral along the path $L_{\gamma_{i+1},4}-L_{\gamma_{i},4}$, one can observe that the function involved in the integrand does not depend on the index $i$ considered, for this function is well defined for $(\epsilon,\tau)\in (D(0,r_{0})\setminus\{0\})\times\dot{D}_{\hat{R}_{\beta}}$. One can apply Cauchy Theorem to derive
$$\int_{L} W_{\beta}(\tau,\epsilon)e^{-\frac{t\tau}{\epsilon}} d\tau=0,$$
where $L= L_{\gamma_{i+1},4}-L_{\gamma_{i},4}-L_{1}$ is the closed path with $s\in[\gamma_{i},\gamma_{i+1}]\to L_1(s)=R_{\beta}e^{\sqrt{-1}s}$. Moreover, $\left|\int_{L_{\gamma_{i+1},4}-L_{\gamma_{i},4}}W_{\beta}(\tau,\epsilon)e^{-\frac{t\tau}{\epsilon }} d\tau\right|$ equals 
\begin{align}
\left|\int_{L_1} W_{\beta}(\tau,\epsilon)e^{-\frac{t\tau}{\epsilon}} d\tau\right|&\le C_{23}\beta!\left(\frac{C_{24}}{\delta}\right)^{\beta}|\epsilon|^{-C\beta}e^{M\log^2(R_{\beta}+\delta_1)}q^{A_{1}\beta^{2}}R_{\beta}\int_{\gamma_{i}}^{\gamma_{i+1}}e^{-\frac{R_{\beta}|t|\cos(\theta-\arg(t/\epsilon))}{|\epsilon|}}d\theta\nonumber\\
&\le\hat{C}_{23}\beta!\left(\frac{C_{24}}{\delta}\right)^{\beta}|\epsilon|^{-C\beta}q^{A_{1}\beta^2}R_{\beta}\int_{\gamma_{i}}^{\gamma_{i+1}}e^{-\frac{R_{\beta}r_{\mathcal{T}}\delta_2}{2|\epsilon|}}ds\nonumber\\
&\le \breve{C}_{23}\beta!\left(\frac{C_{24}}{\delta}\right)^{\beta}q^{A_{1}\beta^2}|\epsilon|^{-C\beta}e^{-\frac{R_{\beta}r_{\mathcal{T}\delta_2}}{2|\epsilon|}}e^{-\frac{R_{\beta}r_{\mathcal{T}\delta_2}}{2|\epsilon|}}.\label{e633}
\end{align}
for some $\hat{C}_{23},\breve{C}_{23}>0$. It only rests to take into account that the function $x\in(0,r_0)\mapsto x^{-C\beta}e^{-\frac{R_{\beta}r_{\mathcal{T}}\delta_2}{2x}}$ is monotone increasing in $(0,r_0)$, so that $|\epsilon|^{-C\beta}e^{-\frac{R_{\beta}r_{\mathcal{T}}\delta_2}{2|\epsilon|}}$ can be included in the constants $\breve{C}_{23}$ and $C_{24}$.

From (\ref{e600}) and (\ref{e633}) one gets the existence of positive constants $C_{6},C_{7}$ such that 
$$\left|X_{i+1,\beta}(\epsilon,t)-X_{i,\beta}(\epsilon,t)\right|\le C_{6}\beta!C_{7}^{\beta}q^{A_{1}\beta^{2}}e^{-\frac{d_1\delta_2 r_{\mathcal{T}}}{2}\frac{q^{d_{2}\beta}}{|\epsilon|}},$$
for every $(\epsilon,t)\in (\mathcal{E}_{i}\cap\mathcal{E}_{i+1})\times \mathcal{T}$. Taking this last estimate into the expression of $X_{i+1}-X_{i}$ one can conclude that
$$|X_{i+1}(\epsilon,t,z)-X_{i}(\epsilon,t,z)|\le C_{6}\sum_{\beta\ge0}(C_{7}\rho)^{\beta}q^{A_{1}\beta^{2}}e^{-\frac{d_1\delta_2r_{\mathcal{T}}}{2}\frac{q^{d_{2}\beta}}{|\epsilon|}},$$
for every $(\epsilon,t,z)\in (\mathcal{E}_{i}\cap\mathcal{E}_{i+1})\times \mathcal{T}\times D(0,\rho)$. The proof of the second statement in the theorem leans on the incoming lemma whose proof is left until the end of the current section. It provides information on the estimates for a Dirichlet type series. A similar argument concerning a Dirichlet series of different nature can be found in \cite{threefold}, Lemma~9, when dealing with Gevrey asymptotic expansions.

\begin{lemma}\label{lema646}
Let $A_{1},D_{1},D_{2},d_2$ be positive constants, with $D_2>1$. Then, for every $\Delta>1$ there exist $E_{1}>0$ and $\delta>0$ such that 
\begin{equation}\label{e648}
\sum_{\beta\ge0}D_{1}^{\beta}q^{A_{1}\beta^2}e^{-D_2\frac{q^{d_2\beta}}{\epsilon}}\le E_{1}e^{-\frac{A_{1}}{d_2^2\Delta}\frac{1}{(-\log(q))2}\log^{2}\epsilon},  
\end{equation}
for every $\epsilon\in (0,\delta]$.
\end{lemma}

\end{proof}

The proof of Lemma~\ref{lema646} heavily rests on the $q-$Gevrey version of some preliminary results which are classical in Gevrey case (see~\cite{threefold} and the references therein). Their proofs do not differ from the classical ones, so we omit them.
\begin{lemma}\label{lema655}
Let $b>0$ and $f:[0,b]\to\C$ a continuous function having the formal expansion $\sum_{n\ge0}a_{n}t^{n}\in\C[[t]]$ as its q-asymptotic expansion of type $A_{1}>0$ at 0, meaning there exist $C,H>0$ such that
$$\left|f(t)-\sum_{n=0}^{N-1}a_{n}t^{n}\right|\le CH^{N}q^{-\frac{A_{1}N^{2}}{2}}\frac{|t|^{N}}{N!},$$
for every $N\ge1$ and $t\in[0,\delta]$, for some $0<\delta<b$.

Then, the function
$$I(x)=\int_{0}^{b}f(s)e^{-\frac{s}{x}}ds$$
admits the formal power series $\sum_{n\ge0}a_{n}n!\epsilon^{n+1}\in\C[[\epsilon]]$ as its $q-$Gevrey asymptotic expansion of type $A_{1}$ at 0. It is to say, there exist $\tilde{C},\tilde{H}>0$ such that
$$\left|I(x)-\sum_{n=0}^{N-1}a_{n}n!x^{n+1}\right|\le \tilde{C}\tilde{H}^{N+1}q^{-\frac{A_{1}(N+1)^{2}}{2}}\frac{|x|^{N+1}}{(N+1)!},$$
for every $N\ge 0$ and $x\in[0,\delta']$ for some $0<\delta'<b$.
\end{lemma}

One can adapt the proof of Proposition 4 in \cite{lastramalek} in our framework. 
\begin{lemma}\label{lema668}
Let $A_{1},\delta>0$ and $\psi:[0,\delta]\to\C$ be a continuous function. Then,
\begin{enumerate}
\item If there exist $C,H>0$ such that $|\psi(x)|\le CH^{n}q^{-\frac{A_{1}n^2}{2}}\frac{|x|^{n}}{n!}$, for every $n\in\N$, $n\ge 0$ and $x\in[0,\delta]$, then for every $\tilde{A}_{1}>A_{1}$ there exists $\tilde{C}>0$ such that 
$$|\psi(x)|\le \tilde{C}e^{-\frac{1}{\tilde{A}_{1}}\frac{1}{(-\log(q))2}\log^{2}|x|}  ,$$
for every $x\in(0,\delta]$.
\item If there exists $C>0$ such that $|\psi(x)|\le Ce^{-\frac{1}{A_{1}}\frac{1}{(-\log(q))2}\log^{2}|x|}$, for every $n\in\N$, and $x\in[0,\delta]$, then for every $\tilde{A}_{1}>A_{1}$ there exists $\tilde{C},\tilde{H}>0$ such that 
$$|\psi(x)|\le \tilde{C}\tilde{H}^{n}q^{-\frac{\tilde{A}_{1}n^2}{2}}\frac{|x|^{n}}{n!} ,$$
for every $n\in\N$ and for every $x\in(0,\delta]$.
\end{enumerate}
\end{lemma}

\textit{proof of Lemma~\ref{lema646}:}

Let $f:[0,+\infty)\to\R$ be a $\mathcal{C}^{1}$ function. For every $n\in\N$, one can apply Euler-Mac-Laurin formula
$$\sum_{\kappa=0}^{n}f(\kappa)=\frac{1}{2}(f(0)+f(n))+\int_{0}^{n}f(t)dt+\int_{0}^{n}B_{1}(t-\left\lfloor t\right\rfloor)f'(t)dt,$$
where $B_{1}(s)=s-\frac{1}{2}$ is the Bernoulli polynomial and $\left\lfloor \cdot\right\rfloor$ stands for the floor function, to $f(s)=D_1^{s}q^{A_{1}s^2}e^{-D_{2}\frac{q^{d_2s}}{\epsilon}}$. One leads to 
\begin{align}
\sum_{\kappa=0}^{n}D_1^{\kappa}q^{A_{1}\kappa^2}e^{-D_{2}\frac{q^{d_2\kappa}}{\epsilon}}&=\frac{1}{2}(e^{-\frac{D_2}{\epsilon}}+D_1^{n}q^{A_{1}n^2}e^{-D_{2}\frac{q^{d_2n}}{\epsilon}})+\int_{0}^{n}D_1^{t}q^{A_{1}t^2}e^{-D_{2}\frac{q^{d_2t}}{\epsilon}}dt\nonumber\\
&+\int_{0}^{n}B_{1}(t-\left\lfloor t\right\rfloor)D_1^{t}q^{A_{1}t^2}e^{-D_{2}\frac{q^{d_2t}}{\epsilon}}(\log(D_1)+\log(q)A_{1}2t-D_2\frac{\log(q)d_2}{\epsilon})dt.\label{e688}
\end{align}
Taking the limit when $n$ tends to infinity in the previous expression we arrive at an equality for a convergent series:
\begin{align*}
\sum_{\kappa=0}^{\infty}D_1^{\kappa}q^{A_{1}\kappa^2}e^{-D_{2}\frac{q^{d_2\kappa}}{\epsilon}}&=\frac{1}{2}e^{-\frac{D_2}{\epsilon}}+\int_{0}^{\infty}D_1^{t}q^{A_{1}t^2}e^{-D_{2}\frac{q^{d_2t}}{\epsilon}}dt\\
&+\int_{0}^{\infty}B_{1}(t-\left\lfloor t\right\rfloor)D_1^{t}q^{A_{1}t^2}e^{-D_{2}\frac{q^{d_2t}}{\epsilon}}(\log(D_1)+\log(q)A_{1}2t-D_2\frac{\log(q)d_2}{\epsilon})dt.
\end{align*}

Let $I_{1}:=\int_{0}^{\infty}f(t)dt$, and $I_{2}:=\int_{0}^{\infty}B_{1}(t-\left\lfloor t\right\rfloor)\left(\log(D_1)+\log(q)A_{1}2t-D_2\frac{\log(q)d_2}{\epsilon}q^{\frac{D_{2}t}{\epsilon}}\right)f(t)dt$.

From the fact that $B_{1}(t-\left\lfloor t\right\rfloor)\le 1/2$ for every $t\ge0$ and the change of variable $D_{2}q^{d_2t}=u$, one gets

$$I_1=\int_{0}^{D_2}f_{1}(u)e^{-\frac{u}{\epsilon}}du,\quad \hbox{where }f_1(u):=D_{1}^{\frac{\log(u/D_2)}{\log(q)d_2}}q^{\frac{A_{1}\log^{2}(u/D_2)}{\log^{2}(q)d_2^2}}\frac{1}{d_{2}(-\log(q))u}$$

and

$$I_2\le \frac{1}{2}\int_{0}^{D_2}(f_2(u)+f_3(u)+f_4(u))e^{-\frac{u}{\epsilon}}du,$$

with $f_2(u):=\log(D_1)f_{1}(u)$, $f_{3}(u):=\frac{2A_{1}}{d_2}\log(u/D_2)f_1(u)$, $f_{4}:=\frac{(-\log(q))d_2}{\epsilon}e^{\frac{D_2\log(u/D_2)}{d_2\epsilon}}f_1(u)$,
for every $u\in(0,D_2]$. Bearing in mind that $f_3(u)<0$ for $u\in(0,D_2]$, and from usual estimates we derive
$I_2\le C_1\left(1+\frac{1}{\epsilon}\right)\int_{0}^{D_2}f_1(u)e^{-\frac{u}{\epsilon}}du$
for some $C_1>0$. The proof is complete if one can estimate $e^{-D_2/\epsilon}$, $I_1$ and $1/\epsilon I_1$ appropriately. The first expression is clearly upper bounded according to (\ref{e648}). 

From usual estimates we arrive at 
$$I_1\le C_3\int_{0}^{D_2}e^{-\frac{A_1}{(-\log(q))d_2^2}\log^2(u/D_2)}e^{-\frac{u}{\epsilon}}du=C_3\int_{0}^{D_3}\tilde{f}(u)e^{-\frac{u}{\epsilon}},$$
for some $C_3>0$. From Lemma~\ref{lema668}, the function $\tilde{g}$ defined by $u\in[0,1]\mapsto\tilde{f}(D_2u)$ (extended by continuity to $u=0$) is such that for every $\tilde{\Delta}>1$ there exist $\tilde{C},\tilde{H}>0$ with 
$$|\tilde{g}(u)|\le \tilde{C}\tilde{H}^{n}q^{-\frac{d_2^2\tilde{\Delta}2}{A_{1}}n^2}\frac{|u|^{n}}{n!},$$
for every $u\in[0,D_2]$ and for every $n\ge0$. From Lemma~\ref{lema655}, the functions $I_{1}(\epsilon)$ and $\frac{1}{\epsilon}I_{1}(\epsilon)$ admit the series with null coefficients as $q-$asymptotic expansion of type $(d_2^2\tilde{\Delta})/A_{1}$. Again, from Lemma~\ref{lema668}, one can conclude that for every $\Delta>\tilde{\Delta}$ there exists $C_4>0$ such that both $I_1$ and $\frac{1}{\epsilon}I_{1}$ are upper bounded by $C_{4}e^{-\frac{A_{1}}{d_2^2\Delta}\frac{1}{(-\log(q))2}\log^2(\epsilon)}$, for every $\epsilon\in(0,\epsilon_1]$, for some $\epsilon_1>0$. 

\hfill $\Box$

\subsection{Existence of formal series solutions in the complex parameter}\label{subseccion33}

In this last subsection we obtain a $q-$Gevrey version of a Malgrange-Sibuya type Theorem. A result in this direction has already been obtained by the authors in~\cite{lastramalek} when dealing with $q\in\C$, $|q|>1$. In that work, the geometry of the problem differs from the one in the present work. Indeed, the result is settled in terms of discrete $q-$spirals tending to the origin, and with $q\in\C$.

Given $q\in\C$ with $0<|q|<1$ and a nonempty open subset $U\subset\C^{\star}$, the discrete $q-$spiral associated to $U$ and $q$ consists of the products of an element in $U$ and $q^{m}$, for some $m\in\N$. For our purpose, $q$ is a real number and $U$ is chosen in such a way that the discrete $q-$spiral turns out to be a sector with vertex at the origin.

The proof of the $q-$Gevrey version of Malgrange-Sibuya Theorem in~\cite{lastramalek} is based on the use of extension results on ultradifferential spaces of weighted functions which preserve the information of $q-$Gevrey bounds but causes that the $q-$Gevrey type involved in the $q-$Gevrey asymptotic suffers an increasement. Here, one can follow similar steps as for the classical proof Malgrange-Sibuya theorem based on Cauchy-Heine transform, so that the $q-$Gevrey type is preserved. In~\cite{malek3}, an analogous demonstration for the Gevrey version of the result can be found. We have decided to include the whole proof of the result  in order to facilitate comprehension and clarity of the result.

\begin{theo}(q-MS)\label{teoremams}\\

Let $(\mathbb{E},||.||_{\mathbb{E}})$ be a Banach space over $\mathbb{C}$ and
$\{ \mathcal{E}_{i} \}_{0 \leq i \leq \nu-1}$ be a good covering in $\mathbb{C}^{\ast}$. For all $0 \leq i \leq \nu-1$, let
$G_{i}$ be a holomorphic function from $\mathcal{E}_{i}$ into
the Banach space $(\mathbb{E},||.||_{\mathbb{E}})$ and let the cocycle $\Delta_{i}(\epsilon) = G_{i+1}(\epsilon) - G_{i}(\epsilon)$ be a
holomorphic function from the sector $Z_{i} = \mathcal{E}_{i+1} \cap \mathcal{E}_{i}$ into $\mathbb{E}$
(with the convention that $\mathcal{E}_{\nu} = \mathcal{E}_{0}$ and $G_{\nu} = G_{0}$). We make the following assumptions.\medskip

\noindent {\bf 1)} The functions $G_{i}(\epsilon)$ are bounded as $\epsilon \in \mathcal{E}_{i}$ tends to the origin in $\mathbb{C}$, for
all $0 \leq i \leq \nu - 1$.\medskip

\noindent {\bf 2)} 
$\Delta_{i}$ has a $q-$exponential decreasing of some type $L>0$, for every $0\le i\le \nu-1$, meaning there exists $C_{i}>0$ such that
\begin{equation}\label{e739}
\left\|\Delta_{i}(\epsilon)\right\|_{\mathbb{E}} \le C_{i}e^{-\frac{1}{L}\frac{1}{(-\log(q))2}\log^{2}|\epsilon|},
\end{equation}
for every $\epsilon\in\mathcal{E}_{i}\cap\mathcal{E}_{i+1}$, and $0\le i\le \nu-1$.

Then, there exists a formal power series
$\hat{G}(\epsilon) \in \mathbb{E}[[\epsilon]]$ such that $G_{i}(\epsilon)$ admits $\hat{G}(\epsilon)$ as its $q-$Gevrey asymptotic expansion of type $L$ on $\mathcal{E}_{i}$, for every $0 \leq i \leq \nu-1$.
\end{theo}
\begin{proof}

We first state an auxiliary result.

\begin{lemma}\label{lema885} For all $0 \leq i \leq \nu-1$, there exist bounded holomorphic functions
$\Psi_{i} : \mathcal{E}_{i} \rightarrow \mathbb{C}$ such that
\begin{equation}
\Delta_{i}(\epsilon) = \Psi_{i+1}(\epsilon) - \Psi_{i}(\epsilon)
\end{equation}
for all $\epsilon \in Z_{i}$, where by convention $\Psi_{\nu}(\epsilon) = \Psi_{0}(\epsilon)$. Moreover, there exist
 $\varphi_{m} \in \mathbb{E}$, $m \geq 0$, such that for each $0 \leq l \leq \nu-1$, any $\hat{L}>L$ and every $\mathcal{W} \prec \mathcal{E}_{l}$, there exist $\hat{K}_{l},\hat{M}_{l}>0$ with
\begin{equation}
 || \Psi_{l}(\epsilon) - \sum_{m=0}^{M-1} \varphi_{m} \epsilon^{m} ||_{\mathbb{E}} \leq
\hat{K}_{l}(\hat{M}_{l})^{M}q^{-\hat{L}\frac{(M-1)^2}{2}} \frac{|\epsilon|^{M}}{M!}
\end{equation}
for all $\epsilon \in \mathcal{W}$, all $M \geq 2$.
\end{lemma}
\begin{proof}

We follow analogous arguments as in Lemma XI-2-6 from \cite{hssi} with appropriate modifications in the asymptotic expansions of the functions constructed with the help of the Cauchy-Heine transform.

For all $0 \leq l \leq \nu-1$, we choose a segment
$$ \mathcal{C}_{l} = \{ te^{\sqrt{-1}\theta_{l}} , t \in [0,r] \} \subset
\mathcal{E}_{l} \cap \mathcal{E}_{l+1}. $$
These $\nu$ segments divide the open punctured disc $D(0,r) \setminus \{ 0 \}$ into $\nu$ open
sectors $\tilde{\mathcal{E}}_{0},\ldots,\tilde{\mathcal{E}}_{\nu-1}$ where
$$ \tilde{\mathcal{E}}_{l} = \{ \epsilon \in \mathbb{C}^{\ast} / \theta_{l-1} <
\mathrm{arg}(\epsilon) < \theta_{l}, |\epsilon| < r \} \ \ , \ \ 0 \leq l \leq \nu-1, $$
where by convention $\theta_{-1}=\theta_{\nu-1}$.
Let
$$ \Psi_{l}(\epsilon) = \frac{-1}{2\pi \sqrt{-1}} \sum_{h=0}^{\nu-1} \int_{C_{h}}
\frac{\Delta_{h}(\xi)}{\xi - \epsilon} d\xi $$
for all $\epsilon \in \tilde{\mathcal{E}}_{l}$, for $0 \leq l \leq \nu-1$, be defined as a sum of Cauchy-Heine transforms of the
functions $\Delta_{h}(\epsilon)$. By deformation of the paths
$C_{l-1}$ and $C_{l}$ without moving their endpoints and letting
the other paths $C_{h}$, $h \neq l-1,l$ untouched (with the convention that
$C_{-1} = C_{\nu-1}$), one can continue analytically the
function $\Psi_{l}$ onto $\mathcal{E}_{l}$. Therefore,
$\Psi_{l}$ defines a holomorphic function on $\mathcal{E}_{l}$, for all $0 \leq l \leq \nu-1$.

Now, take $\epsilon \in \mathcal{E}_{l} \cap \mathcal{E}_{l+1}$. In order to compute
$\Psi_{l+1}(\epsilon) - \Psi_{l}(\epsilon)$, we write
\begin{multline}
\Psi_{l}(\epsilon) = \frac{-1}{2\pi \sqrt{-1}}
\int_{\hat{C}_{l}} \frac{\Delta_{l}(\xi)}{\xi - \epsilon} d\xi
+\frac{-1}{2\pi \sqrt{-1}} \sum_{h=0,h \neq l}^{\nu-1} \int_{C_{h}}
\frac{\Delta_{h}(\xi)}{\xi - \epsilon} d\xi,\\
\Psi_{l+1}(\epsilon) = \frac{-1}{2\pi \sqrt{-1}} \int_{\check{C}_{l}}
\frac{\Delta_{l}(\xi)}{\xi - \epsilon} d\xi
+\frac{-1}{2\pi \sqrt{-1}} \sum_{h=0,h \neq l}^{\nu-1} \int_{C_{h}}
\frac{\Delta_{h}(\xi)}{\xi - \epsilon} d\xi
\end{multline}
where the paths $\hat{C}_{l}$ and $\check{C}_{l}$ are obtained by deforming
the same path $C_{l}$ without moving its endpoints in such a way that:\\
(a) $\hat{C}_{l} \subset \mathcal{E}_{l} \cap \mathcal{E}_{l+1}$ and
$\check{C}_{l} \subset \mathcal{E}_{l} \cap \mathcal{E}_{l+1}$,\\
(b) $\Gamma_{l,l+1} := -\check{C}_{l} + \hat{C}_{l}$ is a simple closed curve with
positive orientation whose interior contains $\epsilon$.

Therefore, due to the residue formula, we can write
\begin{equation}
\Psi_{l+1}(\epsilon) - \Psi_{l}(\epsilon) =  \frac{1}{2\pi \sqrt{-1}} \int_{\Gamma_{l,l+1}}
\frac{\Delta_{l}(\xi)}{\xi - \epsilon} d\xi = \Delta_{l}(\epsilon) \label{difference_psi_l_equal_Delta}
\end{equation}
for all $\epsilon \in \mathcal{E}_{l} \cap \mathcal{E}_{l+1}$, for all $0 \leq l \leq \nu-1$ (with the convention that
$\Psi_{\nu} = \Psi_{0}$).

In a second step, we derive asymptotic properties of $\Psi_{l}$. We fix an $0 \leq l \leq \nu-1$
and a proper closed sector $\mathcal{W}$ contained in $\mathcal{E}_{l}$. Let $\tilde{C}_{l}$ (resp.
$\tilde{C}_{l-1}$) be a path obtained by deforming $C_{l}$ (resp. $C_{l-1}$) without
moving the endpoints in order that $\mathcal{W}$ is contained in the interior of the simple
closed curve $\tilde{C}_{l-1} + \gamma_{l} - \tilde{C}_{l}$ (which is itself contained in
$\mathcal{E}_{l}$), where
$\gamma_{l}$ is a circular arc joining the two points $re^{\sqrt{-1}\theta_{l-1}}$ and
$re^{\sqrt{-1}\theta_{l}}$. We get the representation
\begin{multline}
 \Psi_{l}(\epsilon) =  \frac{-1}{2\pi \sqrt{-1}}
\int_{\tilde{C}_{l}} \frac{\Delta_{l}(\xi)}{\xi - \epsilon} d\xi
+  \frac{-1}{2\pi \sqrt{-1}} \int_{\tilde{C}_{l-1}} \frac{\Delta_{l-1}(\xi)}{\xi - \epsilon} d\xi\\
+ \frac{-1}{2\pi \sqrt{-1}} \sum_{h=0,h \neq l,l-1}^{\nu-1} \int_{C_{h}}
\frac{\Delta_{h}(\xi)}{\xi - \epsilon} d\xi
\end{multline}
for all $\epsilon \in \mathcal{W}$. One assumes that the path $\tilde{C}_{l}$ is given as the
union of a segment $L_{l} = \{ te^{\sqrt{-1}w_{l}} : t \in [0,r_{1}] \}$ where $r_{1} < r$ and
$w_{l} > \theta_{l}$ and a curve $\Gamma_{l} = \{ \mu_{l}(\tau) : \tau \in [0,1] \}$
such that $\mu_{l}(0) = r_{1}e^{\sqrt{-1}w_{l}}$, $\mu_{l}(1)=re^{\sqrt{-1}\theta_{l}}$
and $r_{1} \leq |\mu_{l}(\tau)| < r$ for all $\tau \in [0,1)$. We also assume that there exists a
positive number $\sigma<1$ with $|\epsilon| \leq \sigma r_{1}$ for all $\epsilon \in \mathcal{W}$.
By construction of the path $\Gamma_{l}$, we get that the function
$\epsilon \mapsto \frac{1}{2\pi \sqrt{-1}}
\int_{\Gamma_{l}}\frac{\Delta_{l}(\xi)}{\xi - \epsilon} d\xi$ defines
an analytic function on the open disc $D(0,r_{1})$.

It remains to give estimates for the integral $\frac{1}{2\pi \sqrt{-1}} \int_{L_{l}}
\frac{\Delta_{l}(\xi)}{\xi - \epsilon} d\xi$. Let $M \geq 0$ be an integer. From the usual geometric
series expansion, one can write
\begin{equation}
\frac{1}{2\pi \sqrt{-1}} \int_{L_{l}}
\frac{\Delta_{l}(\xi)}{\xi - \epsilon} d\xi = \sum_{m=0}^{M} \alpha_{l,m}\epsilon^{m} +
\epsilon^{M+1}E_{l,M+1}(\epsilon) \label{A_E_int_Ll_Delta}
\end{equation}
where
\begin{equation}
\alpha_{l,m} = \frac{1}{2 \pi \sqrt{-1}} \int_{L_{l}} \frac{\Delta_{l}(\xi)}{\xi^{m+1}}
d\xi \ \ , \ \ E_{l,M+1}(\epsilon) = \frac{1}{2 \pi \sqrt{-1}} \int_{L_{l}}
\frac{\Delta_{l}(\xi)}{\xi^{M+1}(\xi-\epsilon)} d\xi \label{alpha_m_and_EM_integral}
\end{equation}
for all $\epsilon \in \mathcal{W}$.

Gathering (\ref{e739}) and (\ref{alpha_m_and_EM_integral}), we get 

\begin{equation}\label{norm_E_alpha_lm_2<}
||\alpha_{l,m}||_{\mathbb{E}} \leq \frac{K_{l}}{2\pi} \int_{0}^{r_1}
\frac{e^{-\frac{1}{L}\frac{1}{(-\log(q))2}\log^2\tau  }}{\tau^{m+1}}
d\tau 
\end{equation}
The changes of variable $\log(\tau)=s$ first, and $s=\sqrt{2L(-\log(q))}t$ afterwards, transform the right-hand side of (\ref{norm_E_alpha_lm_2<}) into
\begin{align*}
 \frac{K_{l}}{2\pi} \int_{-\infty}^{\log(r_1)}
\frac{e^{-\frac{1}{L}\frac{1}{(-\log(q))2}s^2}}{e^{sm}}ds &=  \frac{K_{l}\sqrt{2L(-\log(q))}}{2\pi} \int_{-\infty}^{\frac{\log(r_1)}{\sqrt{2L(-\log(q))}}}
\exp(-t^2-m\sqrt{2L(-\log(q))}t)dt\\
&\le \frac{K_{l}\sqrt{2L(-\log(q))}}{2\pi} \int_{-\infty}^{\infty}
\exp(-t^2-m\sqrt{2L(-\log(q))}t)dt.
\end{align*}
 
The application of 
$$e^{\frac{a^2}{4}}\sqrt{\pi}=\int_{-\infty}^{\infty}e^{-x^2-ax}dx,$$
for every $a\in\R$, which can be found in~\cite{andrews} (Chapter 10, p. 498), leads us to
\begin{equation}\label{e1007}
||\alpha_{l,m}||_{\mathbb{E}}\le \frac{\sqrt{2L(-\log(q))}K_l}{2\sqrt{\pi}}q^{-L\frac{m^2}{2}}.
\end{equation}

Moreover, as above, one can choose a positive
number $\eta > 0$ (depending on $\mathcal{W}$) such that $|\xi - \epsilon| \geq |\xi|\sin(\eta)$
for all $\xi \in L_{l}$ and all $\epsilon \in \mathcal{W}$. Again by (\ref{e739}) and (\ref{alpha_m_and_EM_integral}), and following analogous calculations as before we obtain 
\begin{equation}
 ||E_{l,M+1}(\epsilon)||_{\mathbb{E}} \leq \frac{K_{l}}{2\pi \sin(\eta)}
\int_{0}^{r_1} \frac{e^{-\frac{1}{L}\frac{1}{2(-\log(q))}\log^2\tau}}{\tau^{M+2}} d\tau \leq
\frac{\sqrt{2L(-\log(q))}K_l}{2\sin(\nu)\sqrt{\pi}}q^{-L\frac{(M+1)^2}{2}}
\label{EM<fact_M_log_M}
\end{equation}
for all $\epsilon \in \mathcal{W}$.
Using comparable arguments, one can give analogous estimates when estimating the other integrals
$$
 \frac{-1}{2\pi \sqrt{-1}} \int_{\tilde{C}_{l-1}} \frac{\Delta_{l-1}(\xi)}{\xi - \epsilon} d\xi \ \ , \ \
 \frac{-1}{2\pi \sqrt{-1}} \int_{C_{h}} \frac{\Delta_{h}(\xi)}{\xi - \epsilon} d\xi
$$
for all $h \neq l,l-1$.

As a consequence, for any $0 \leq l \leq \nu-1$, there exist $\varphi_{l,m} \in \mathbb{E}$, for all $m \geq 0$ and a constant
$\hat{K}_{l}>0$ such that
\begin{equation}
||\Psi_{l}(\epsilon) - \sum_{m=0}^{M-1} \varphi_{l,m}\epsilon^{m}||_{\mathbb{E}} \leq
\hat{K}_{l}q^{-L\frac{M^2}{2}}|\epsilon|^{M}
\label{A_E_psi_l_I_2}
\end{equation}
for all $M \geq 2$, all $\epsilon \in \mathcal{W}$.

Taking into account Proposition~\ref{prop587}, we deduce that for every $\hat{\mathcal{E}}_{i,i+1}\prec\mathcal{E}_{l}\cap\mathcal{E}_{l+1}$ and for every $\hat{L}>L$, the function $\Psi_{l+1}(\epsilon)-\Psi_{l}(\epsilon)$ has the
formal series $\hat{0}$ as $q-$Gevrey asymptotic expansion of type $\hat{L}$ in $\hat{\mathcal{E}}_{i,i+1}$. From the unicity of the asymptotic
expansions on sectors, we deduce that all the formal series
$\sum_{m \geq 0} \varphi_{l,m} \epsilon^{m}$, $0 \leq l \leq \nu-1$, are equal
to some formal series denoted
$\hat{G}(\epsilon) = \sum_{m \geq 0} \varphi_{m}\epsilon^{m} \in \mathbb{E}[[\epsilon]]$.

\end{proof}

We consider now the bounded holomorphic functions
$$ a_{i}(\epsilon) = G_{i}(\epsilon) - \Psi_{i}(\epsilon) $$
for all $0 \leq i \leq \nu-1$, all $\epsilon \in \mathcal{E}_{i}$. By definition, for any $i\in \{0,...,\nu-1\}$, we have that
$$ a_{i+1}(\epsilon) - a_{i}(\epsilon) = G_{i+1}(\epsilon) - G_{i}(\epsilon) - \Delta_{i}(\epsilon) = 0 $$
for all $\epsilon \in Z_{i}$. Therefore, each $a_{i}(\epsilon)$ is the restriction on $\mathcal{E}_{i}$ of a holomorphic function
$a(\epsilon)$ on $D(0,r) \setminus \{ 0 \}$. Since $a(\epsilon)$ is moreover bounded on $D(0,r) \setminus \{ 0 \}$, the origin turns out
to be a removable singularity for $a(\epsilon)$ which, as a consequence, defines a convergent power series on $D(0,r)$.

Finally, one can write 
$$ G_{i}(\epsilon) = a(\epsilon) + \Psi_{i}(\epsilon) $$
for all $\epsilon \in \mathcal{E}_{i}$, all $0 \leq i \leq \nu-1$. Moreover, $a(\epsilon)$ is a convergent power series, and for every $\hat{L}>L$,
$\Psi_{i}(\epsilon)$ has the series $\hat{G}(\epsilon) = \sum_{m \geq 0} \varphi_{m} \epsilon^{m}$ as
$q-$Gevrey asymptotic expansion of type $\hat{L}$ on $\mathcal{E}_{i}$, for all $0 \leq i \leq \nu-1$.

\end{proof}

We are under conditions to enunciate the main result in the present work.

\begin{theo}\label{teorema813}
Let $\rho>0$. 

Under the same hypotheses as in Theorem~\ref{teorema549}, we denote $\mathbb{H}_{\mathcal{T},\rho}$ the Banach space of holomorphic and bounded functions in $\mathcal{T}\times D(0,\rho)$ with the supremum norm. Then, there exists a formal power series 
$$\hat{X}(\epsilon,t,z)=\sum_{k\ge0}\frac{X_{k}(t,z)}{k!}\epsilon^{k}\in\mathbb{H}_{\mathcal{T},\rho}[[\epsilon]],$$
formal solution of
$$ \epsilon \partial_{t}\partial_{z}^{S}\hat{X}(\epsilon,t,z)+a\partial_{z}^{S}\hat{X}(\epsilon,t,z)=\sum_{\underline{\kappa}=(\kappa_{0},\kappa_{1})\in\mathcal{N}}b_{\underline{\kappa}}(\epsilon,z)(\partial_{t}^{\kappa_{0}}\partial_{z}^{\kappa_{1}}\hat{X})(\epsilon,q^{m_{\underline{\kappa},1}}t,q^{m_{\underline{\kappa},2}}z).$$
Moreover, for every $0\le i\le \nu-1$ and every $L_2>\frac{d_2^2}{A_1}$, the function $X_{i}(\epsilon,t,z)$ constructed in Theorem~\ref{teorema549} admits $\hat{X}(\epsilon,t,z)$ as its $q-$Gevrey asymptotic expansion of type $L_2$ in $\mathcal{E}_{i}$, meaning that for every $0\le i\le \nu-1$ and $\tilde{\mathcal{E}}_{i}\prec \mathcal{E}_{i}$, there exist $L_0,L_1>0$ such that 
\begin{equation}\label{e839}
\sup_{\stackrel{t\in\mathcal{T}}{z\in D(0,\rho)}}\left|X(\epsilon,t,z)-\sum_{k=0}^{N}\frac{X_{k}(t,z)}{k!}\epsilon^{k}\right|\le L_0L_1^{N}q^{-L_2\frac{N^{2}}{2}}\frac{|\epsilon|^{N+1}}{(N+1)!},
\end{equation}
for every $N\ge0$ and all $\epsilon\in\tilde{\mathcal{E}}_{i}$.
\end{theo}

\begin{proof}
Let us consider the family $(X_{i}(\epsilon,t,z))_{0\le i\le \nu-1}$ constructed in Theorem~\ref{teorema549}. For every $0\le i\le \nu-1$, we define the function $\epsilon\in\mathcal{E}_i\mapsto G_{i}(\epsilon):=X_{i}(\epsilon,t,z)$, which belongs to the space $\mathbb{E}_{\mathcal{T},\rho}$. From (\ref{e552}), we derive the cocycle $\Delta_{i}:=G_{i+1}(\epsilon)-G_{i}(\epsilon)$ verifies (\ref{e739}) in Theorem~\ref{teoremams}, with $1/L:=\frac{A_{1}}{d_2^2\Delta}$ for some fixed $0<\Delta<1$. Theorem~\ref{teoremams} guarantees the existence of a formal power series $\hat{X}(\epsilon)\in\mathbb{H}_{\mathcal{T},\rho}[[\epsilon]]$, such that for every $L_2>L$, $G_{i}(\epsilon)$ admits $\hat{G}(\epsilon)$ as its $q-$Gevrey asymptotic expansion of type $	L_{2}$ on $\mathcal{E}_{i}$. This is valid for every $0\le i\le \nu-1$.  This concludes the second part of the result.

It only rests to verify that $\hat{X}$ is a formal solution of (\ref{e497})+(\ref{e502}).

If we write $\hat{X}(\epsilon,t,z):=\sum_{k\ge0}\frac{X_{k}(t,z)}{k!}\epsilon^{k}$, we have 
\begin{equation}\label{e768}
\lim_{\stackrel{\epsilon\to 0}{\epsilon\in\tilde{\mathcal{E}}_{i}}}\sup_{(t,z)\in\mathcal{T}\times D(0,\rho)}\left|\partial_{\epsilon}^{\ell}X_{i}(t,z,\epsilon)-X_{\ell}(t,z)\right|=0,
\end{equation}
for every $0\le i\le \nu-1$ and all $\ell\ge0$.

Let $0\le i \le \nu-1$. By construction, $X_{i}$ satisfies (\ref{e497})+(\ref{e502}). We differenciate in the equality (\ref{e497}) $\ell\ge 1$ times with respect to $\epsilon$. By Leibniz's rule, we deduce that $\partial_{\epsilon}^{\ell}X_{i}(\epsilon,t,z)$ satisfies
\begin{align*}
&\epsilon\partial_{\epsilon}^{\ell}\partial_{t}\partial_{z}^{S}X_{i}(\epsilon,t,z)+\ell\partial_{\epsilon}^{\ell-1}\partial_{t}\partial_{z}^{S}X_{i}(\epsilon,t,z)+a\partial_{z}^{S}X_{i}(\epsilon,t,z)\\
&=\sum_{\underline{\kappa}=(\kappa_0,\kappa_1)\in\mathcal{N}}\sum_{\ell_0+\ell_1=\ell}\frac{\ell!}{\ell_0!\ell_1!}\partial_{\epsilon}^{\ell_{1}}b_{\underline{\kappa}}(\epsilon,z)(\partial_{\epsilon}^{\ell_1}\partial_{t}^{\kappa_0}\partial_{z}^{\kappa_{1}}X_{i})(\epsilon,q^{m_{\underline{\kappa},1}}t,q^{m_{\underline{\kappa},2}}z),
\end{align*}
for every $(\epsilon,t,z)\in\mathcal{E}_{i}\times \mathcal{T}\times D(0,\rho)$. Let $\epsilon\to 0$ in the previous expression. From (\ref{e768}) we obtain
\begin{align}&\partial_{t}\partial_{z}^{S}\left(\frac{X_{\ell-1}(t,z)}{(\ell-1)!}\right)+a\partial_{z}^{S}\left(\frac{X_{\ell}(t,z)}{\ell!}\right)\nonumber\\
&=\sum_{\underline{\kappa}=(\kappa_0,\kappa_1)\in\mathcal{N}}\sum_{\ell_0+\ell_1=\ell} \left(\frac{\partial_{\epsilon}^{\ell_0}b_{\underline{\kappa}}(0,z)}{\ell_0!}\right)   \left(\frac{\partial_{t}^{\kappa_0}\partial_{z}^{\kappa_1}X_{\ell_1}(q^{m_{\underline{\kappa},1}}t,q^{m_{\underline{\kappa},2}}z)}{\ell_1!}\right).\label{e780}     
\end{align}
$b_{\underline{\kappa}}(z,\epsilon)$ is holomorphic wih respect to $\epsilon$ for every $\underline{\kappa}\in\mathcal{N}$. This entails $b_{\underline{\kappa}}(\epsilon,z)=\sum_{h\ge0}\frac{\partial_{\epsilon}^{h}b_{\underline{\kappa}}(0,z)}{h!}\epsilon^{h},$
for every $(\epsilon,z)$ in a neighborhood of the origin in $\C^{2}$. From this, and (\ref{e780}), we deduce $\hat{X}(\epsilon,t,z)=\sum_{k\ge0}\frac{X_{k}(t,z)}{k!}\epsilon^{k}\in\mathbb{H}_{\mathcal{T},\rho}[[\epsilon]]$ is a formal solution of (\ref{e497})+(\ref{e502}).
\end{proof}

\end{section}

\end{document}